\documentclass[11pt, reqno]{amsart}
\usepackage{amsthm,amsmath,amssymb}
\usepackage{mathtools}
\usepackage{graphicx}
\usepackage{caption}
\usepackage{ragged2e}
\usepackage{mathrsfs}
\usepackage{hyperref}
\usepackage{nicematrix}

\allowdisplaybreaks
\graphicspath{{Figure/}}
\numberwithin{equation}{section}
\numberwithin{figure}{section}
\theoremstyle{plain}

\newtheorem{thm}{Theorem}[section]
\newtheorem{lem}[thm]{Lemma}

\newtheorem{prop}[thm]{Proposition}

\theoremstyle{definition}
\theoremstyle{remark}

\newcommand{\M}{\operatorname{M}}
\newcommand{\Q}{\operatorname{Q}}

\newcommand{\wt}{\operatorname{wt}}
\newcommand{\Pf}{\operatorname{Pf}}
\newcommand{\GF}{\operatorname{GF}}
\newcommand{\sgn}{\operatorname{sgn}}

\begin{document}
\setlength{\abovedisplayskip}{10pt}
\setlength{\belowdisplayskip}{10pt}

\title[A reflection principle for nonintersecting paths]{A reflection principle for nonintersecting paths and lozenge tilings with free boundaries}

\author{Seok Hyun Byun}

\address{Department of Mathematics, Amherst College}
\email{sbyun@amherst.edu}

\begin{abstract}
Okada and Stembridge's Pfaffian formula for the enumeration of families of nonintersecting paths with fixed starting points and unfixed ending points has been widely used to resolve many challenging problems in enumerative combinatorics. In this paper, we present a new formula that complements Okada and Stembridge's Pfaffian formula. The proof is based on a formula for the square of the sum of maximum minors of matrices obtained from Okada's formula. The combinatorial interpretation of the new formula gives a reflection principle for nonintersecting paths. It implies that the enumeration of families of nonintersecting paths with unfixed ending points can be resolved by enumerating families of nonintersecting paths with fixed ending points instead. Using this formula, we also show that the enumeration of lozenge tilings of a large family of regions with free boundaries can be deduced from those without free boundaries. We then provide several applications of this result, including 1) a new family of regions whose tiling generating function is given by a simple product formula, 2) a simpler proof of a factorization theorem for lozenge tilings of hexagons with holes, and 3) new determinant formulas for the volume generating functions of shifted plane partitions of a shifted shape and symmetric plane partitions of a symmetric shape.
\end{abstract}

\maketitle

\section{Introduction}\label{section1}

Enumeration of families of nonintersecting paths is one of the important topics in enumerative combinatorics, as nonintersecting paths are in bijection with many other combinatorial objects. Among other techniques, the classical theorems of Lindstr\"{o}m--Gessel--Viennot (see \cite{L} and \cite{GV1,GV2}) and Okada--Stembridge (see \cite{O1} and \cite{Ste}) have successfully answered many enumeration problems. We recall the statement of these important theorems.

Let $G$ be a locally finite and acyclic directed graph. Assume that for each edge $e$ of $G$, we assign weight $\wt(e)$ from some commutative ring (usually, it is $\mathbb{R}$ or a power series ring). We also assign weight $1$ to a degenerate edge whose two adjacent vertices are the same. For any two vertices $u$ and $v$, consider a directed path\footnote{In this paper, all paths are directed. Thus, we simply call them \textit{paths} throughout the paper.} $P$ directed from $u$ to $v$. The weight of the path $P$, $\wt(P)$, is the product of the weights of all edges that constitute the path $P$. Given an $m$-tuple of paths $(P_1,\ldots,P_m)$, the weight of the $m$-tuple of paths is the product of weight of the $m$ paths, i.e. $\prod_{i=1}^{m}\wt(P_i)$. We say that the $m$-tuple of paths $(P_1,\ldots,P_m)$ is \textit{nonintersecting} if $P_i$ and $P_j$ do not intersect for all $i\neq j$. For two vertices $u$ and $v$ of $G$, let $\mathscr{P}(u,v)$ be the set of all directed paths from $u$ to $v$ and $\GF[\mathscr{P}(u,v)]$ be the sum of weight of all the paths $P\in\mathscr{P}(u,v)$. In particular, when $u=v$, $\mathscr{P}(u,u)$ consists of a single path of length zero whose weight is one (thus, $\GF[\mathscr{P}(u,u)]=1$). For positive integers $m$ and $n$ such that $m\leq n$, consider an $m$-tuple of vertices $\mathbf{u}=(u_1,\ldots,u_m)$ and an $n$-tuple of vertices $\mathbf{v}=(v_1,\ldots,v_n)$ of $G$. We say that $\mathbf{u}$ and $\mathbf{v}$ are \textit{compatible} if for every\footnote{For positive integers $k$, $[k]\coloneqq\{1,\ldots,k\}$.} $i,j\in[m]$ and $k,l\in[n]$ such that $i<j$ and $k<l$, every path $P\in\mathscr{P}(u_i,v_l)$ intersects with every path $Q\in\mathscr{P}(u_j,v_k)$. For any permutation $\sigma\in S_{m}$, let $\mathscr{P}_0(\mathbf{u}_{\sigma}, \mathbf{v})$ be the set of all $m$-tuples of nonintersecting paths $(P_1,\ldots,P_m)$ such that $P_k\in\mathscr{P}(u_{\sigma(k)}, v_{j_k})$ for some indices $j_1,\ldots,j_m$ satisfying $j_1<\cdots<j_m$. Note that if $\mathbf{u}$ and $\mathbf{v}$ are compatible, then $\mathscr{P}_0(\mathbf{u}_{\sigma}, \mathbf{v})$ is empty unless $\sigma$ is an identity permutation. When $\sigma$ is an identity permutation, we simply denote $\mathscr{P}_0(\mathbf{u}_{\sigma}, \mathbf{v})$ by $\mathscr{P}_0(\mathbf{u}, \mathbf{v})$ rather than $\mathscr{P}_0(\mathbf{u}_{id}, \mathbf{v})$ . The sum of weight of all $m$-tuples of nonintersecting paths in $\mathscr{P}_0(\mathbf{u}_{\sigma}, \mathbf{v})$ is denoted by $\GF[\mathscr{P}_0(\mathbf{u}_{\sigma}, \mathbf{v})]$. Lastly, let $\M(\mathbf{u,\mathbf{v}})$ be the $m\times n$ matrix defined by
\begin{equation}\label{eaa}
    \M(\mathbf{u},\mathbf{v})=\Big[\GF[\mathscr{P}(u_i,v_j)]\Big]_{1\leq i\leq m, 1\leq u\leq n}
\end{equation}
and we call this matrix the \textit{path matrix from $\mathbf{u}$ to $\mathbf{v}$}. The pictures that illustrate the cases when $m=n$ and $m<n$ are given in Figure \ref{faa}.

To state Lindstr\"{o}m--Gessel--Viennot theorem (see \cite{L} and \cite{GV1,GV2}), we assume $n=m$. Under this assumption,
\begin{itemize}
    \item $\mathscr{P}_0(\mathbf{u}_{\sigma}, \mathbf{v})$ is the set of all $m$-tuples of nonintersecting paths $(P_1,\ldots,P_m)$ such that $P_k\in\mathscr{P}(u_{\sigma(k)}, v_{k})$ and
    \item $\M(\mathbf{u},\mathbf{v})$ becomes an $m\times m$ square matrix.
\end{itemize}

\begin{thm}[Lindstr\"{o}m, Gessel and Viennot]\label{taa}
    Let $G$ be a locally finite and acyclic directed graph and $\mathbf{u}=(u_1,\ldots,u_m)$, $\mathbf{v}=(v_1,\ldots,v_m)$ are two $m$-tuple of vertices on $G$. Then,
    \begin{equation}\label{eab}
        \sum_{\sigma\in S_m}\sgn(\sigma)\GF[\mathscr{P}_{0}(\mathbf{u}_{\sigma},\mathbf{v})]=\det \big[\M(\mathbf{u},\mathbf{v})\big].
    \end{equation}
 
    In particular, if $\mathbf{u}$ and $\mathbf{v}$ are compatible, then
    \begin{equation}\label{eac}
        \GF[\mathscr{P}_{0}(\mathbf{u},\mathbf{v})]=\det \big[\M(\mathbf{u},\mathbf{v})\big].
    \end{equation}
\end{thm}

\begin{figure}
    \centering
    \includegraphics[width=0.75\textwidth]{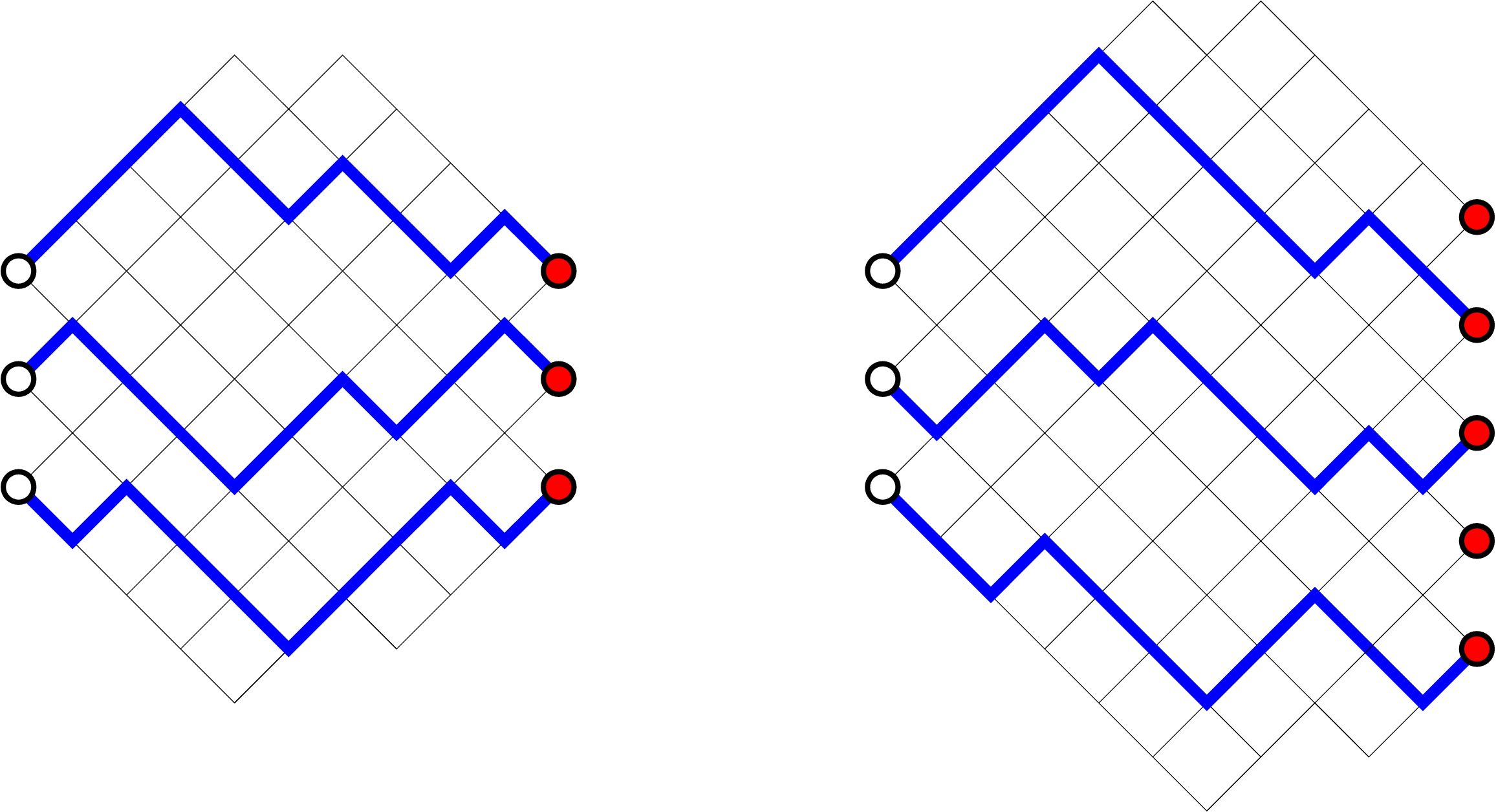}
    \caption{Directed graphs with three starting points and three ending points (left) and three starting points and five ending points (right) together with three nonintersecting paths on them. In the pictures, all the edges are directed from left to right, and the starting and ending points are marked by empty and filled circles, respectively.}
    \label{faa}
\end{figure}

We now state Okada--Stembridge's generalization of Theorem \ref{taa} (see \cite{O1} and \cite{Ste}). What we state below is a slight extension of their theorem, which can be found in the paper of Ciucu and Krattenthaler (see \cite{CK}). We replace the assumption $n=m$ by $n\geq m$ and assume further that $m$ is even. 

\begin{thm}[Okada--Stembridge]\label{tab}
    Let $G$ be a locally finite and acyclic directed graph and $\mathbf{u}=(u_1,\ldots,u_m)$ and $\mathbf{v}=(v_1,\ldots,v_n)$ be an $m$-tuple and an $n$-tuple of vertices on $G$, respectively, where $m\leq n$ and $m$ is even. Then\footnote{The Pfaffian of a skew-symmetric matrix $A$, denoted by $\Pf[A]$, will be defined in Section 2.},
    \begin{equation}\label{ead}
        \sum_{\sigma\in S_m}\sgn(\sigma)\GF[\mathscr{P}_0(\mathbf{u}_{\sigma},\mathbf{v})]=\Pf\big[\Q(\mathbf{u},\mathbf{v})\big]
    \end{equation}
    where $\Q(\mathbf{u},\mathbf{v})=\big[Q_{i,j}\big]_{1\leq i,j\leq m}$ is the skew-symmetric matrix defined by 
    \begin{equation}\label{eae}
    Q_{i,j}=\sum_{s<t}\Big[\GF[\mathscr{P}(u_i,v_s)]\cdot \GF[\mathscr{P}(u_j,v_t)]-\GF[\mathscr{P}(u_i,v_t)]\cdot \GF[\mathscr{P}(u_j,v_s)]\Big].
    \end{equation}
    
    In particular, if $\mathbf{u}$ and $\mathbf{v}$ are compatible, then
    \begin{equation}\label{eaf}
        \GF[\mathscr{P}_0(\mathbf{u},\mathbf{v})]=\Pf\big[\Q(\mathbf{u},\mathbf{v})\big].
    \end{equation}
\end{thm}
Stembridge pointed out\footnote{Stembridge made this comment for the case when $\mathbf{u}$ and $\mathbf{v}$ are compatible, but this comment can also be applied even when $\mathbf{u}$ and $\mathbf{v}$ are not compatible.} in \cite{Ste} that if $m$ is odd, one can still use the above theorem using the following idea: add a phantom vertex $u_0=v_0$ to the graph $G$, with no incident edges and apply the theorem above using $\mathbf{u}'=(u_0,u_1,\ldots,u_m)$ and $\mathbf{v}'=(v_0,v_1,\ldots,v_n)$ instead of $\mathbf{u}$ and $\mathbf{v}$. Since $u_0$ can only be connected with $v_0$ using a single path of length zero, which has weight $1$, one can show that\footnote{On the right side of the equality, $S_{m+1}$ needs to be understood as the set of permutations of the set $\{0,1,\ldots,m\}$ rather than that of the set $[m+1]$. The equality holds because the summand $\GF[\mathscr{P}_0(\mathbf{u}'_{\sigma'},\mathbf{v}')]$ vanishes unless $\sigma'(0)=0$.} $\sum_{\sigma\in S_m}\sgn(\sigma)\GF[\mathscr{P}_0(\mathbf{u}_{\sigma},\mathbf{v})]=\sum_{\sigma'\in S_{m+1}}\sgn(\sigma')\GF[\mathscr{P}_0(\mathbf{u}'_{\sigma'},\mathbf{v}')]$. Since $\mathbf{u}'$ consists of $(m+1)$ vertices and $(m+1)$ is even, the above theorem gives a Pfaffian expression for the right side of the equation, and thus, we obtain a Pfaffian expression for $\sum_{\sigma\in S_m}\sgn(\sigma)\GF[\mathscr{P}_0(\mathbf{u}_{\sigma},\mathbf{v})]$ this way. One can notice that, unlike when $m$ is even, the Pfaffian expression for $m$ odd obtained this way has order $(m+1)$.

One natural question arising after comparing Theorems \ref{taa} and \ref{tab} is the following.
\begin{itemize}
    \item While there is no parity condition on $m$ (the number of starting points) in Theorem \ref{taa}, $m$ should be even to apply Theorem \ref{tab} directly. One can still apply Theorem \ref{tab} when $m$ is odd, following Stembridge's remark above, but the Pfaffian has a different order than when $m$ is even. Is there an alternative formula to \eqref{ead} and \eqref{eaf} that does not depend on the parity of $m$?
\end{itemize}

The following theorem, which is the main theorem of this paper, gives an affirmative answer to the question above.

\begin{thm}\label{tac}
Let $G$ be a locally finite and acyclic directed graph and $\mathbf{u}=(u_1,\ldots,u_m)$ and $\mathbf{v}=(v_1,\ldots,v_n)$ be an $m$-tuple and an $n$-tuple of vertices on $G$, respectively, where $m\leq n$. Then,
    \begin{equation}\label{eag}
        \Bigg[\sum_{\sigma\in S_m}\sgn(\sigma)\GF[\mathscr{P}_0(\mathbf{u}_{\sigma},\mathbf{v})]\Bigg]^2=\det\Big[\M(\mathbf{u},\mathbf{v})U_{n}\M(\mathbf{u},\mathbf{v})^T\Big]=\det\Big[\M(\mathbf{u},\mathbf{v})U_{n}^T\M(\mathbf{u},\mathbf{v})^T\Big],
    \end{equation}
    where $\M(\mathbf{u},\mathbf{v})$ is the path matrix and $U_{n}=[u_{i,j}]_{1\leq i,j\leq n}$ is the upper triangular matrix defined by $u_{i,j}=
    \begin{cases}
    2, & \text{if $i<j$}\\
    1, & \text{if $i=j$}\\    
    0, & \text{if $i>j$}
    \end{cases}$.
    
    In particular, if $\mathbf{u}$ and $\mathbf{v}$ are compatible, then
    \begin{equation}\label{eah}
        \GF[\mathscr{P}_0(\mathbf{u},\mathbf{v})]^2=\det\Big[\M(\mathbf{u},\mathbf{v})U_{n}\M(\mathbf{u},\mathbf{v})^T\Big]=\det\Big[\M(\mathbf{u},\mathbf{v})U_{n}^T\M(\mathbf{u},\mathbf{v})^T\Big].
    \end{equation}
\end{thm}
As mentioned earlier, the formulas \eqref{eag} and \eqref{eah} in Theorem \ref{tac} give positive answers to the question stated above. More precisely, there is no parity condition on $m$ in Theorem \ref{tac}. Thus, we have a uniform formula that works for both $m$ even and odd. Furthermore, the formulas in \eqref{eag} and \eqref{eah} are natural extensions of \eqref{eab} and \eqref{eac} because if $m=n$, then $\M(\mathbf{u},\mathbf{v})$ becomes a square matrix and we have $\det\Big[\M(\mathbf{u},\mathbf{v})U_{n}\M(\mathbf{u},\mathbf{v})^T\Big]=\det\Big[\M(\mathbf{u},\mathbf{v})U_{n}^T\M(\mathbf{u},\mathbf{v})^T\Big]=\Big[\det\big[\M(\mathbf{u},\mathbf{v})\big]\Big]^2$.

Throughout the paper, we provide several applications of Theorem \ref{tac}. One of them is a combinatorial interpretation of \eqref{eag} and \eqref{eah}, which we call a \textit{reflection principle for nonintersecting paths}, stated in Theorem \ref{tbe}. It states that under some mild assumptions, to enumerate families of nonintersecting paths with fixed starting points and unfixed ending points on a graph, one can instead enumerate families of nonintersecting paths with the same number of starting and ending points (thus the ending points are also fixed) on a new graph that one can obtain from the original graph by ``almost" symmetrizing it. Analogous theorem for lozenge tilings of regions with free boundaries and other applications will also be presented. 

The paper is organized as follows.
\begin{itemize}
    \item In Section \ref{section2}, we first recall the definition of Pfaffian and some properties of skew-symmetric matrices (see Proposition \ref{tba}). We then prove a lemma (Lemma \ref{tbb}) and recall Okada's result on the sum of maximum minors of rectangular matrices in Theorem \ref{tbc}. Using the lemma and Okada's formula, we then prove a formula for the square of the sum of maximum minors in Theorem \ref{tbd}. Using this theorem, we prove Theorem \ref{tac}. We finish the section by stating and proving a reflection principle for nonintersecting paths in Theorem \ref{tbe}.
    \item In Section \ref{section3}, we apply Theorem 1.3 to lozenge tilings enumeration problems. We prove Theorem \ref{tca}, which states that the enumeration of lozenge tilings of a large family of regions with free boundaries can be resolved by instead enumerating the lozenge tilings of their counterpart regions that do not have any free boundaries. Interestingly, the counterpart regions can be obtained by ``almost" symmetrizing the regions with free boundaries.
    \item In Section \ref{section4}, we give applications of Theorems \ref{tca} and \ref{tac}, including 1) a new family of regions whose tiling generating function is given by a simple product formula (see Theorem \ref{tda}), 2) a simpler proof of a factorization theorem for lozenge tilings of hexagons with holes (see Theorem \ref{tdb}), and 3) new determinant formulas for the volume generating functions of shifted plane partitions of a shifted shape and symmetric plane partitions of a symmetric shape (see Theorem \ref{tdd}).
\end{itemize}

\section{The sum of maximum minors and a reflection principle for nonintersecting paths}\label{section2}

A $2m\times2m$ matrix $A=[a_{i,j}]_{1\leq i,j\leq 2m}$ is \textit{skew symmetric} if $a_{i,j}+a_{j,i}$=0 for all $i,j\in[2m]$. To define a \textit{Pfaffian} of a skew-symmetric matrix $A$, denoted by $\Pf_{2m} (A)$ or $\Pf (A)$ (if the size of the matrix $A$ is clear from the context), we first introduce some notions. Consider a complete graph $K_{2m}$ and label its vertices by the elements of $[2m]=\{1,\ldots,2m\}$ such that different vertices have different labels. If an edge of $K_{2m}$ is adjacent to vertices labeled by $i$ and $j$ (with $i<j$), then we denote the edge by $(i,j)$. A \textit{1-factor} (or a \textit{perfect matching}) of $K_{2m}$ is a collection of $m$ edges of $K_{2m}$ that covers every vertex exactly once. We denote the set of 1-factor of $K_{2m}$ by $\mathscr{F}_{m}$. We say that two edges $(i,j)$ and $(k,l)$ of $\pi\in\mathscr{F}_{m}$ are crossing if either $i<k<j<l$ or $k<i<l<j$ holds. The crossing number of $\pi$, denoted by $\operatorname{cr}(\pi)$, is the number of crossed pairs in $\pi$ and the sign of the 1-factor $\pi\in\mathscr{F}_{m}$, denoted by $\sgn(\pi)$, is $\sgn(\pi)\coloneqq(-1)^{\operatorname{cr}(\pi)}$. Then the Pfaffian of a skew-symmetric matrix $A=[a_{i,j}]_{1\leq i,j\leq 2m}$ is defined as follows.
\begin{equation}\label{eba}
    \Pf(A)=\sum_{\pi\in\mathscr{F}_{m}}\sgn(\pi)\prod_{(i,j)\in\pi}a_{i,j}.
\end{equation}

Among many properties of skew-symmetric matrices, we recall some of them that we will use later in this paper. They are stated and proved in \cite{Ste} (see Propositions 2.2 and 2.3 in that paper).

\begin{prop}\label{tba}
    For any skew-symmetric matrix $A=[a_{i,j}]_{1\leq i,j\leq 2m}$,
    \begin{enumerate}
        \item $\Pf(A)^2=\det(A)$.
        \item $\det[a_{i,j}+tx_ix_j]=\det[a_{i,j}]=\det(A)$.
    \end{enumerate}
\end{prop}

To prove the main theorem, the following simple lemma is needed. For any skew symmetric matrix $A=[a_{i,j}]$ and an indeterminate $x$, let $A(x)$ be the matrix defined by $A(x)\coloneqq[a_{i,j}+x]$. 

\begin{lem}\label{tbb}
    Let $m$, $n$, and $k$ be nonnegative integers such that $m+k$ is even. Let $Z=[z_{i,j}]_{1\leq i\leq m, 1\leq j\leq n}$ be an $m\times n$ matrix, $A=[a_{i,j}]_{1\leq i,j\leq n}$ be an $n\times n$ skew symmetric matrix, $H$ be an $m\times k$ matrix, and $B$ be an $k\times k$ skew symmetric matrix. Then, we have
    \begin{equation}\label{ebb}
        \det
        \begin{bmatrix}
            ZAZ^T & H\\\
            -H^T & B
        \end{bmatrix}
        =\det
        \begin{bmatrix}
            ZA(x)Z^T & H\\\
            -H^T & B
        \end{bmatrix}.
    \end{equation}
    In particular, if $k=0$, then $m$ is even and we have
    \begin{equation}\label{ebc}
        \det\Big[ZAZ^T\Big]=\det\Big[ZA(x)Z^T\Big].
    \end{equation}
\end{lem}
\begin{proof}
    Since $A$ and $B$ are skew-symmetric matrices, $ZAZ^T$ is also a skew-symmetric matrix, and so is the block matrix
    $\begin{bmatrix}
            ZAZ^T & H\\\
            -H^T & B
    \end{bmatrix}$.
    The $(i,j)$-entry of $ZA(x)Z^T$ is
    \begin{equation}\label{ebd}
        \sum_{k,l=1}^{n}(a_{k,l}+x)z_{i,k}z_{j,l}=\sum_{k,l=1}^{n}a_{k,l}z_{i,k}z_{j,l}+x\Bigg(\sum_{k=1}^{n}z_{i,k}\Bigg)\Bigg(\sum_{k=1}^{n}z_{j,k}\Bigg)
    \end{equation}
    and thus the $(i,j)$-entry of $ZAZ^T=ZA(0)Z^T$ is $\sum_{k,l=1}^{n}a_{k,l}z_{i,k}z_{j,l}$. Therefore, \eqref{ebb} holds by applying Proposition \ref{tba} (b) to
    $\begin{bmatrix}
            ZAZ^T & H\\\
            -H^T & B
    \end{bmatrix}$
    by setting $t=x$ and $x_i=
    \begin{cases}
        \sum_{k=1}^{n}z_{i,k}, & \text{if $1\leq i\leq m$}\\
            0, & \text{otherwise}
    \end{cases}.
    $
    This completes the proof.
\end{proof}

For positive integers $m$ and $n$, consider an $m\times n$ matrix $Z=(z_{i,j})_{1\leq i\leq m, 1\leq j\leq n}$. We then define $\widetilde{Z}=(\widetilde{z}_{i,j})_{0\leq i\leq m, 0\leq j\leq n}$ by $\widetilde{z}_{i,j}=\begin{cases}
    1, & \text{if $i=j=0$}\\
    0, & \text{if [$i=0$ and $j>0$] or [$i>0$ and $j=0$]}\\
    z_{i,j}, & \text{otherwise}
\end{cases}$.
Denote the sum of the maximum minors of $Z$ by $\sigma(Z)$. More precisely,
\begin{equation}\label{ebe}
    \sigma(Z)\coloneqq\sum_{J\subseteq[n]}\det Z_{[m],J},
\end{equation}
where $J$ runs over all subsets of $[n]$ whose cardinality is $m$. In \cite{O1}, Okada proved the following theorem, which states that $\sigma(Z)$ is given by a Pfaffian of a certain skew-symmetric matrix. We express Okada's results using Ishikawa and Wakayama's expression in \cite{IW}. Let $E_{n}=[e_{i,j}]_{1\leq i,j\leq n}$ be the skew-symmetric matrix defined by $e_{i,j}=
    \begin{cases}
    1, & \text{if $i<j$}\\
    0, & \text{if $i=j$}\\    
    -1, & \text{if $i>j$}
    \end{cases}$.

\begin{thm}[Theorem 3 in \cite{O1}]\label{tbc}
For any $m\times n$ matrix $Z$ with $m\leq n$, the sum of maximum minors of $Z$, $\sigma(Z)$, satisfies the following.

    (a) If $m$ is odd, then
    \begin{equation}\label{ebf}
        \sigma(Z)=\Pf_{m+1}[\widetilde{Z}E_{n+1}\widetilde{Z}^T].
    \end{equation}
    
    (b) If $m$ is even, then
    \begin{equation}\label{ebg}
        \sigma(Z)=\Pf_{m}[ZE_nZ^T].
    \end{equation}
\end{thm}

We now state and prove a formula for the square of the sum of maximum minors $\sigma(Z)^2$, which does not depend on the parity of $m$. The proof is based on Theorem \ref{tbc} and Lemma \ref{tbb}.

\begin{thm}\label{tbd}
    For any $m\times n$ matrix $Z$ with $m\leq n$, the square of the sum of maximum minors of $Z$, $\sigma(Z)^2$, satisfies the following.
    \begin{equation}\label{ebh}
        \sigma(Z)^2=\det[ZU_nZ^{T}]=\det[ZU_n^{T}Z^{T}],
    \end{equation}
    where $U_{n}=[u_{i,j}]_{1\leq i,j\leq n}$ is the upper triangular matrix defined by $u_{i,j}=
    \begin{cases}
    2, & \text{if $i<j$}\\
    1, & \text{if $i=j$}\\    
    0, & \text{if $i>j$}
    \end{cases}$.
\end{thm}
\begin{proof}
We first consider the case when $m$ is even. By Proposition \ref{tba}, Lemma \ref{tbb}, and \eqref{ebg}, 
\begin{equation}\label{ebi}
    \sigma(Z)^2=\Pf_{m}[ZE_nZ^T]^2=\det[ZE_nZ^T]=\det[ZE_n(1)Z^T]=\det[ZU_nZ^T].
\end{equation}
This completes the proof when $m$ is even. If $m$ is odd, by Proposition \ref{tba}, Lemma \ref{tbb}, and \eqref{ebf},
\begin{equation}\label{ebj}
    \sigma(Z)^2=\Pf_{m+1}[\widetilde{Z}E_{n+1}\widetilde{Z}^T]^2=\det[\widetilde{Z}E_{n+1}\widetilde{Z}^T]=\det[\widetilde{Z}E_{n+1}(1)\widetilde{Z}^T]=\det[\widetilde{Z}U_{n+1}\widetilde{Z}^T].
\end{equation}
Note that $\widetilde{Z}U_{n+1}\widetilde{Z}^T$ is an $(m+1)\times(m+1)$ matrix and for integer $i$, $j$ such that $0\leq i,j\leq m$, its $(i,j)$-entry is
\begin{equation}\label{ebjk}
    [\widetilde{Z}U_{n+1}\widetilde{Z}^T]_{i,j}=
    \begin{cases}
    1, & \text{if $i=j=0$}\\
    0, & \text{if $i>0$ and $j=0$}\\    
    2\sum_{k=1}^{n}z_{j,k}, & \text{if $i=0$ and $j>0$}\\
    [ZU_{n}Z^T]_{i,j}, & \text{otherwise}
    \end{cases}.
\end{equation}
Hence, applying the Laplace expansion along the leftmost column to $\widetilde{Z}U_{n+1}\widetilde{Z}^T$, we get
\begin{equation}\label{ebl}
    \det[\widetilde{Z}U_{n+1}\widetilde{Z}^T]=\det[ZU_{n}Z^T].
\end{equation}
Combining \eqref{ebj} and \eqref{ebl}, we get
\begin{equation}\label{ebm}
    \sigma(Z)^2=\det[ZU_{n}Z^T]
\end{equation}
and this completes the proof for the case when $m$ is odd.
\end{proof}

Using Theorem \ref{tbd}, we give a proof of the main theorem, Theorem \ref{tac}.

\begin{proof}[Proof of Theorem \ref{tac}]
    From Theorem \ref{tbd}, we have
    \begin{equation}\label{ebn}
        \det\Big[\M(\mathbf{u},\mathbf{v})U_{n}\M(\mathbf{u},\mathbf{v})^T\Big]=\sigma(\M(\mathbf{u},\mathbf{v}))^2.
    \end{equation}
    Also, by its definition,
    \begin{equation}\label{ebo}
        \sigma(\M(\mathbf{u},\mathbf{v}))=\sum_{1\leq j_1<\ldots<j_{m}\leq n}\det\M(\mathbf{u},\mathbf{v})_{[m],\{j_1,\ldots,j_m\}}
    \end{equation}
    is the sum of maximum minors of the matrix $\M(\mathbf{u},\mathbf{v})$.
    
    Since $(i,k)$-th entry of the matrix $\M(\mathbf{u},\mathbf{v})_{[m],\{j_1,\ldots,j_m\}}$ is $\GF[\mathscr{P}(u_i,v_{j_k})]$, by Lindstr\"{o}m--Gessel--Viennot theorem (Theorem \ref{taa}),
    \begin{equation}\label{ebp}
        \det\M(\mathbf{u},\mathbf{v})_{[m],\{j_1,\ldots,j_m\}}=\sum_{\sigma\in S_m}\sgn(\sigma)\GF[\mathscr{P}_{0}(\mathbf{u}_{\sigma},(v_{j_1},\ldots,v_{j_{m}}))].
    \end{equation}
    Since
    \begin{equation}\label{ebq}
        \GF[\mathscr{P}_0(\mathbf{u}_{\sigma}, \mathbf{v})]=\sum_{1\leq j_1<\ldots<j_{m}\leq n}\GF[\mathscr{P}_{0}(\mathbf{u}_{\sigma},(v_{j_1},\ldots,v_{j_{m}}))],
    \end{equation}
    combining \eqref{ebn}-\eqref{ebq}, we get
    \begin{equation}\label{ebr}
        \Bigg[\sum_{\sigma\in S_m}\sgn(\sigma)\GF[\mathscr{P}_0(\mathbf{u}_{\sigma},\mathbf{v})]\Bigg]^2=\det\Big[\M(\mathbf{u},\mathbf{v})U_{n}\M(\mathbf{u},\mathbf{v})^T\Big].
    \end{equation}
    The second equality in \eqref{eag} holds because determinants are invariant under taking the transpose. \eqref{eah} is also true because if $\mathbf{u}$ and $\mathbf{v}$ are compatible, then $\mathscr{P}_0(\mathbf{u}_{\sigma},\mathbf{v})$ vanishes unless $\sigma$ is the identity permutation. This completes the proof.
\end{proof}

A combinatorial interpretation of Theorem \ref{tac} gives a reflection principle for nonintersecting paths. As in Theorem \ref{tac}, consider a locally finite and acyclic directed graph $G$ and an $m$-tuple $\mathbf{u}=(u_1,\ldots,u_m)$ and an $n$-tuple $\mathbf{v}=(v_1,\ldots,v_n)$ of vertices on $G$, where $m\leq n$. Assume further that the vertices $v_1,\ldots,v_n$ are sinks (that is, there are no outgoing edges from these $n$ vertices in $G$) and these $n$ vertices are on the infinite face of $G$ in a cyclic order. We then choose a line $l$ that does not intersect with $G$ and reflect $G$ across $l$. Let $G'$ be the mirror image of $G$ under the reflection. For every vertex $x$ in $G$, let $x'$ in $G'$ be the mirror image of $x$ and for every edge $e$ in $G$ directed from a vertex $x$ to a vertex $y$, let $e'$ in $G'$ be its mirror image that is directed from $y'$ to $x'$ and have the same weight as $e$ (that is, $\wt(e)=\wt(e')$). Also, denote the mirror images of an $m$-tuple $\mathbf{u}=(u_1,\ldots,u_m)$ and an $n$-tuple $\mathbf{v}=(v_1,\ldots,v_n)$ by $\mathbf{u'}=(u_1',\ldots,u_m')$ and $\mathbf{v'}=(v_1',\ldots,v_n')$, respectively. We denote the union of $G$ and $G'$ by $G_{sym}$. Note that $G_{sym}$ has two components, $G$ and $G'$. We connect the two components by adding $(3n-2)$ edges between vertices in $\mathbf{v}$ and $\mathbf{v'}$ with unit weight on them in two different ways and construct two graphs $\overline{G_{sym}}$ and $\widetilde{G_{sym}}$. To construct $\overline{G_{sym}}$, from $G_{sym}$, we add
\begin{enumerate}
    \item $n$ edges connecting $v_i$ and $v_i'$ for $i\in[n]$, directed toward $v_i'$,
    \item $(n-1)$ edges connecting $v_i$ and $v_{i+1}'$ for $i\in[n-1]$, directed toward $v_{i+1}'$, and
    \item $(n-1)$ edges connecting $v_i'$ and $v_{i+1}'$ for $i\in[n-1]$, directed toward $v_{i+1}'$
\end{enumerate}
(see the picture on the top in Figure \ref{fba} for an example). We draw these edges so they do not cross other edges and denote the resulting graph by $\overline{G_{sym}}$. If we instead add
\begin{enumerate}
    \item[$(1)'$] $n$ edges connecting $v_i$ and $v_i'$ for $i\in[n]$, directed toward $v_i'$,
    \item[$(2)'$] $(n-1)$ edges connecting $v_{i+1}$ and $v_{i}'$ for $i\in[n-1]$, directed toward $v_{i}'$, and
    \item[$(3)'$] $(n-1)$ edges connecting $v_{i+1}'$ and $v_{i}'$ for $i\in[n-1]$, directed toward $v_{i}'$,
\end{enumerate}
from $G_{sym}$, then the resulting graph is $\widetilde{G_{sym}}$ (for instance, see the bottom picture in Figure \ref{fba}). Note that by their construction and the two assumptions (the vertices in $\mathbf{v}$ are sink of $G$ and they are on the infinite face of $G$ in a cyclic order), both $\overline{G_{sym}}$ and $\widetilde{G_{sym}}$ are also locally finite and acyclic directed graphs.

For any permutation $\sigma\in S_m$, let $\overline{\mathscr{P}_{0}}(\mathbf{u}_{\sigma},\mathbf{u'})$ (or $\widetilde{\mathscr{P}_{0}}(\mathbf{u}_{\sigma},\mathbf{u'})$) be the sets of $m$-tuples of nonintersecting paths $(\overline{P_1},\ldots,\overline{P_m})$ in $\overline{G_{sym}}$ such that $\overline{P_i}\in\mathscr{P}(u_{\sigma(i)},u_i')$ for $i\in[m]$ (or $(\widetilde{P_1},\ldots,\widetilde{P_m})$ in $\widetilde{G_{sym}}$ such that $\widetilde{P_i}\in\mathscr{P}(u_{\sigma(i)},u_i')$ for $i\in[m]$). Also when $\sigma$ is an identity permutation, we write $\overline{\mathscr{P}_{0}}(\mathbf{u},\mathbf{u'})$ (or $\widetilde{\mathscr{P}_{0}}(\mathbf{u},\mathbf{u'})$) instead of $\overline{\mathscr{P}_{0}}(\mathbf{u}_{id},\mathbf{u'})$ (or $\widetilde{\mathscr{P}_{0}}(\mathbf{u}_{id},\mathbf{u'})$).

\begin{figure}
    \centering
    \includegraphics[width=0.6\textwidth]{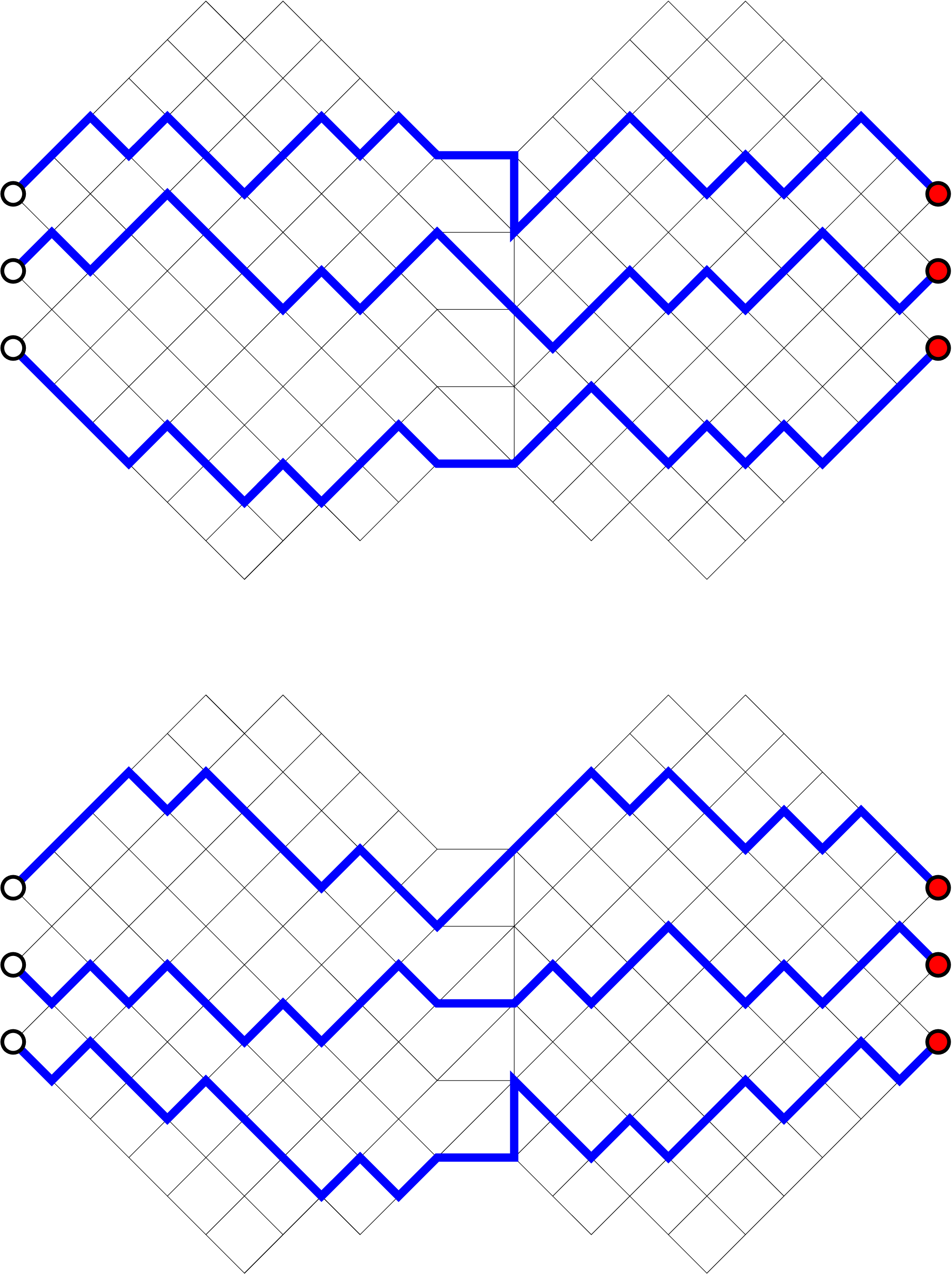}
    \caption{Two graphs with three starting points and three ending points, and three nonintersecting paths on them. In the pictures, the starting and ending points are marked by empty and filled circles, respectively. While all the nonvertical edges are directed from left to right in both pictures, vertical edges are oriented \textit{downward} in the top picture and \textit{upward} in the picture below, respectively. These are examples of $\overline{G_{sym}}$ (top) and $\widetilde{G_{sym}}$ (bottom) appearing in Theorem \ref{tbe}, where $G$ is the right picture in Figure \ref{faa}.}
    \label{fba}
\end{figure}

\begin{thm}[A reflection principle for nonintersecting paths]\label{tbe}
Let $G$, $\mathbf{u}=(u_1,\ldots,u_m)$, and $\mathbf{v}=(v_1,\ldots,v_n)$ be the same as in Theorem \ref{tac}. Assume further that the vertices $v_1,\ldots,v_n$ are sink of $G$ and are on the infinite face of $G$ in a cyclic order so that we can construct $\overline{G_{sym}}$ and $\widetilde{G_{sym}}$ as described in the previous paragraph. Then,
\begin{equation}\label{ebs}
        \Bigg[\sum_{\sigma\in S_m}\sgn(\sigma)\GF[\mathscr{P}_0(\mathbf{u}_{\sigma},\mathbf{v})]\Bigg]^2=\sum_{\sigma\in S_m}\sgn(\sigma)\GF[\overline{\mathscr{P}_0}(\mathbf{u}_{\sigma},\mathbf{u}')]=\sum_{\sigma\in S_m}\sgn(\sigma)\GF[\widetilde{\mathscr{P}_0}(\mathbf{u}_{\sigma},\mathbf{u}')].
    \end{equation}

In particular, if $\mathbf{u}$ and $\mathbf{v}$ are compatible, then
    \begin{equation}\label{ebt}
        \GF[\mathscr{P}_{0}(\mathbf{u},\mathbf{v})]^2=\GF[\overline{\mathscr{P}_{0}}(\mathbf{u},\mathbf{u'})]=\GF[\widetilde{\mathscr{P}_{0}}(\mathbf{u},\mathbf{u'})].
    \end{equation}
\end{thm}
\begin{proof}
We first prove \eqref{ebs}. By Theorem \ref{tac}, it suffices to show 
\begin{equation}\label{ebu}
    \sum_{\sigma\in S_m}\sgn(\sigma)\GF[\overline{\mathscr{P}_0}(\mathbf{u}_{\sigma},\mathbf{u}')]=\det\Big[\M(\mathbf{u},\mathbf{v})U_{n}\M(\mathbf{u},\mathbf{v})^T\Big]
\end{equation}
and
\begin{equation}\label{ebv}
    \sum_{\sigma\in S_m}\sgn(\sigma)\GF[\widetilde{\mathscr{P}_0}(\mathbf{u}_{\sigma},\mathbf{u}')]=\det\Big[\M(\mathbf{u},\mathbf{v})U_{n}^T\M(\mathbf{u},\mathbf{v})^T\Big].
\end{equation}
We first prove \eqref{ebu}. Notice that when the underlying graph is $\overline{G_{sym}}$, we have $U_{n}=\M(\mathbf{v},\mathbf{v}')$ and $\M(\mathbf{v}',\mathbf{u}')=\M(\mathbf{u},\mathbf{v})^T$. The former is true because 
\begin{equation*}
    \mathscr{P}(v_i,v_j')=
\begin{cases}
    2, & \text{if $i<j$: two paths, $v_i\rightarrow v_i'\rightarrow\cdots\rightarrow v_j'$ and $v_i\rightarrow v_{i+1}'\rightarrow\cdots\rightarrow v_j'$}\\
    1, & \text{if $i=j$: one path, $v_i\rightarrow v_i'$}\\    
    0, & \text{if $i>j$: no path from $v_i$ to $v_j'$}
\end{cases}
\end{equation*}
and the latter is true because $\mathbf{v}'$ and $\mathbf{u}'$ in $G'$ are mirror images of $\mathbf{v}$ and $\mathbf{u}$ in $G$, respectively. Thus,
\begin{align*}
    \det\Big[\M(\mathbf{u},\mathbf{v})U_{n}\M(\mathbf{u},\mathbf{v})^T\Big]=\det\Big[\M(\mathbf{u},\mathbf{v})\M(\mathbf{v},\mathbf{v}')\M(\mathbf{v}',\mathbf{u}')\Big]&=\det\big[\M(\mathbf{u},\mathbf{u}')\big]\\
    &=\sum_{\sigma\in S_m}\sgn(\sigma)\GF[\overline{\mathscr{P}_0}(\mathbf{u}_{\sigma},\mathbf{u}')],
\end{align*}
where the last equality is due to the Lindstr\"{o}m--Gessel--Viennot theorem (Theorem \ref{taa}). The proof of \eqref{ebv} is almost the same as that of \eqref{ebu}. The only difference is, since the underlying graph is $\widetilde{G_{sym}}$, we have $U_{n}^{T}=\M(\mathbf{v},\mathbf{v}')$ because
\begin{equation*}
    \mathscr{P}(v_i,v_j')=
\begin{cases}
    0, & \text{if $i<j$: no path from $v_i$ to $v_j'$}\\
    1, & \text{if $i=j$: one path, $v_i\rightarrow v_i'$}\\    
    2, & \text{if $i>j$: two paths, $v_i\rightarrow v_i'\rightarrow\cdots\rightarrow v_j'$ and $v_i\rightarrow v_{i-1}'\rightarrow\cdots\rightarrow v_j'$}
\end{cases}
\end{equation*}
and the rest of the proof is the same as that of \eqref{ebu}.

When $\mathbf{u}$ and $\mathbf{v}$ are compatible, regardless of the underlying graphs (either $\overline{G_{sym}}$ or $\widetilde{G_{sym}}$), it suffices show that $\mathbf{u}$ and $\mathbf{u}'$ are compatible. If we show that $\mathbf{u}$ and $\mathbf{u}'$ are compatible, then \eqref{ebt} follows from \eqref{ebs}. 
When $\mathbf{u}$ and $\mathbf{v}$ compatible, notice that $\mathbf{v'}$ and $\mathbf{u'}$ are also compatible by their construction (recall that $\mathbf{v'}$ and $\mathbf{u'}$ are mirror images of $\mathbf{v}$ and $\mathbf{u}$, respectively). Also, again by construction, $\mathbf{v}$ and $\mathbf{v'}$ are compatible (in both $\overline{G_{sym}}$ and $\widetilde{G_{sym}}$) and thus we can deduce from it that $\mathbf{u}$ and $\mathbf{u'}$ are compatible. This completes the proof of \eqref{ebt}.
\end{proof}

\section{Enumeration of Lozenge tilings of regions with free boundaries}\label{section3}

Consider a triangular lattice whose one family of its lattice lines is vertical. A \textit{lozenge} is a union of two adjacent unit triangles in the lattice. Given a region $R$ on the lattice, a \textit{lozenge tiling} of the region is a collection of lozenges (in the region) that covers the region without gaps and overlaps. If every lozenge $l$ is weighted by $\wt(l)$ and a lozenge tiling of the region is given, the \textit{weight of the tiling} is the product of the weights of all lozenges that constitute the tiling. A \textit{tiling generating function} of the region $R$, denoted by $\M_{\wt}(R)$ (or just $\M(R)$ if the weight function is clear from the context), is the sum of weights of all lozenge tilings of $R$. Note that if all lozenges are weighted by $1$, then $\M(R)$ gives the number of lozenge tilings of $R$. Throughout this paper, each lozenge is assigned weight $1$ unless its weight is defined separately.

In the previous paragraph, although it was not clearly presented, we assumed that all the boundaries of the region $R$ are \textit{solid} boundaries. This means that no lozenge is allowed to cross the boundary. A boundary of the region $R$ is \textit{free} if a lozenge is allowed to cross the boundary. We now define lozenge tilings for regions with free boundaries. Consider a region $R$ and assume that parts of its boundaries are free. In this case, a \textit{lozenge tiling of the region with free boundaries} is a collection of lozenges that covers the region without gaps and overlaps, where some of the lozenges are allowed to cross the part of $R$'s boundary that are free (see the left picture in Figure \ref{fcc}. There, a free boundary is denoted by a dashed line, while solid boundaries are denoted by solid lines). The definition of the tiling generating function of the regions with free boundaries is the same as before. To emphasize the existence of free boundaries, if the region $R$ has free boundaries, we use $\M_f(R)$ (instead of $\M(R)$) to denote its tiling generation function (or the number of its lozenge tilings if all lozenges are weighted by $1$). 

Enumerating the lozenge tilings of regions with free boundaries is much harder than enumerating those without free boundaries. The challenge is, while several methods, including Lindstr\"{o}m--Gessel--Viennot theorem (Theorem \ref{taa}) and Kuo's graphical condensation method \cite{Kuo}, can be used to deal with the regions with no free boundary, one cannot apply most of the techniques if free boundaries exist. When there are free boundaries, it is known that Okada and Stembridge's Pfaffian formula (Theorem \ref{tab}) can be used to deal with this constraint. However, while various techniques for evaluating the Lindstr\"{o}m--Gessel--Viennot type determinants have been developed (for example, see \cite{Kra1} and \cite{Kra2}), evaluating Okada--Stembridge type Pfaffian is less studied and relatively hard. It is worth mentioning that Ciucu \cite{C6} developed a technique that allows one to solve free boundary lozenge tiling problems inductively (by generalizing Kuo's graphical condensation method \cite{Kuo}). However, fewer techniques are still available when free boundaries exist. The main theorem of this section, Theorem \ref{tca}, states that the enumeration of lozenge tilings of a large family of regions with free boundaries can be resolved by considering the enumeration of lozenge tilings of regions without any free boundary instead.

A \textit{partition} $\lambda=(\lambda_1,\ldots,\lambda_k)$ is a sequence\footnote{If $k=0$, then $\lambda$ is an empty sequence.} of weakly decreasing positive integers. A partition is \textit{strict} if it is a sequence of strictly decreasing positive integers. In this paper, a strict partition is denoted by $\lambda_{st}$.

Before we describe the regions of our interest, we first define a $\wedge$-hook and a shifted $\wedge$-hook. For any positive integer $n$, a \textit{$\wedge$-hook of order $n$} is a $\wedge$-shaped hook that consists of $2n$ unit lozenges as described in the left picture in Figure \ref{fca}. A \textit{shifted $\wedge$-hook of order $n$} is also a $\wedge$-shaped hook that is obtained from the $\wedge$-hook of order $n$ by shifting the left-most unit triangle to the right end of it (see the right picture in Figure \ref{fca}).

\begin{figure}
    \centering
    \includegraphics[width=.8\textwidth]{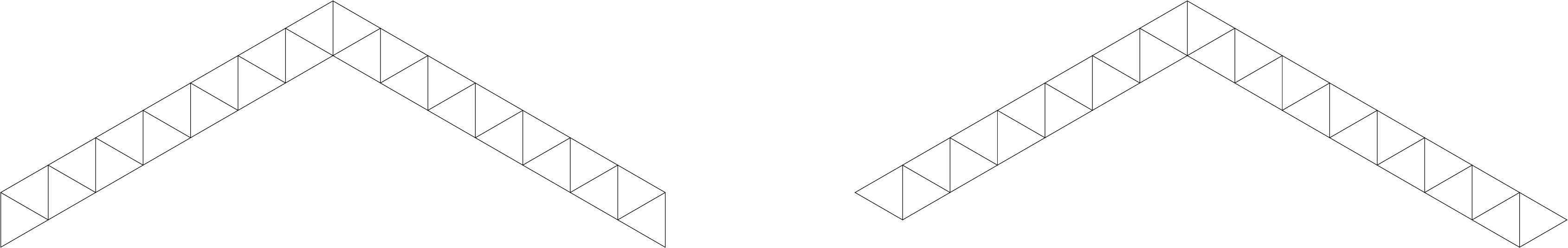}
    \caption{The $\wedge$-hook of order $7$ (left) and the shifted $\wedge$-hook of order $7$ (right).}
    \label{fca}
\end{figure}

In this section, we introduce a family of regions with a straight line free boundary, $A(m;\lambda_{st};I)$, and its counterpart region, $\widetilde{A}(m;\lambda_{st};I)$, that do not have any free boundary, where $m\in\mathbb{Z}_{\geq0}$, $\lambda_{st}$ is a strict partition with $k$ parts, and $I\subseteq[k]$. Using Theorem \ref{tac}, we show that $\M_f(A(m;\lambda_{st};I))$ and $\M(\widetilde{A}(m;\lambda_{st};I))$ satisfy a simple relation (see Theorem \ref{tca}). This result implies that to enumerate the number of lozenge tilings of $A(m;\lambda_{st};I)$, one can instead look at the tiling generating function of $\widetilde{A}(m;\lambda_{st};I)$, which do not have any free boundary.

\begin{figure}
    \centering
    \includegraphics[width=1.0\textwidth]{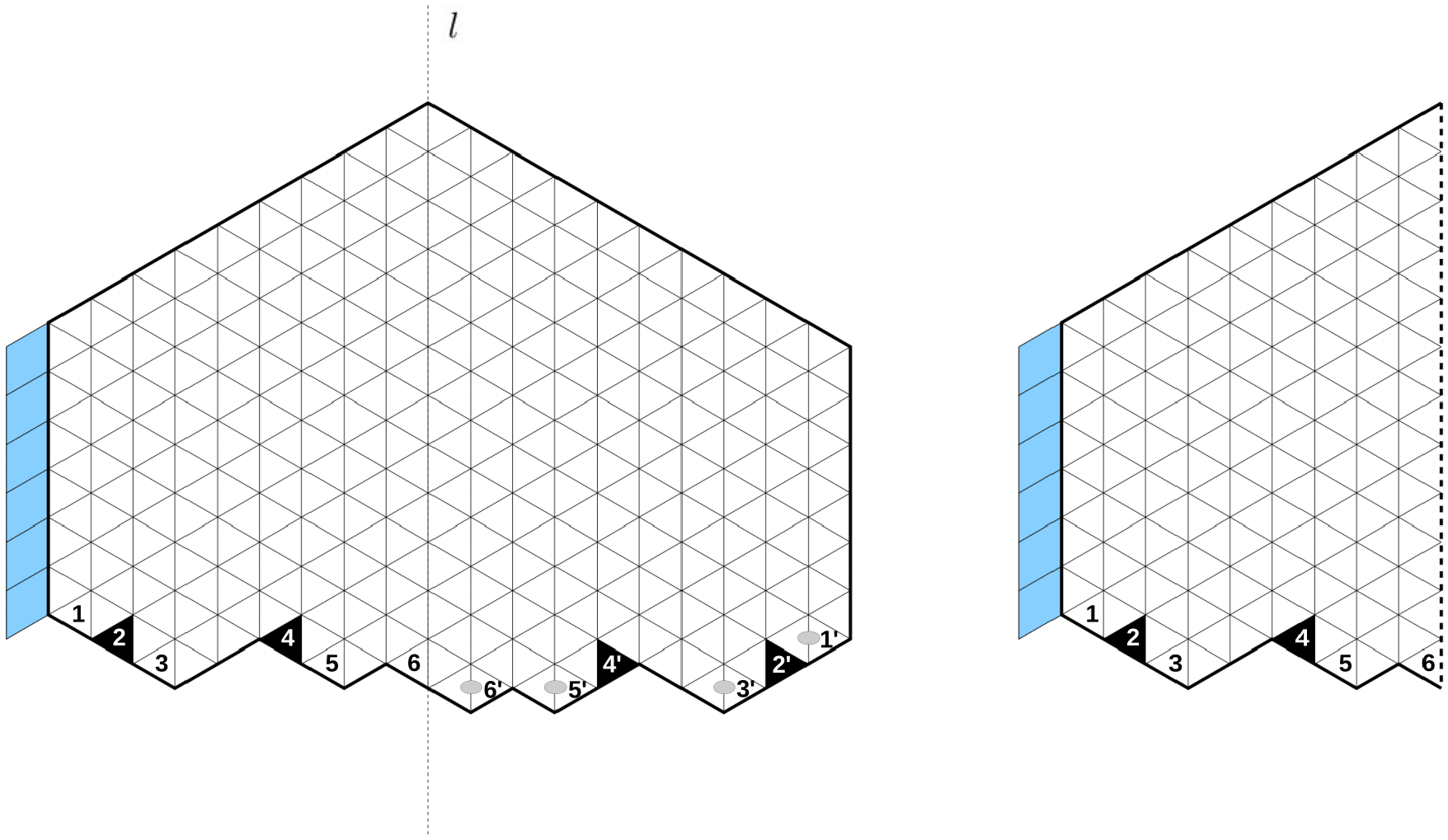}
    \caption{$\widetilde{A}(m;\lambda_{st};I)$ (left) and $A(m;\lambda_{st};I)$ (right) with $m=6$, $\lambda=(9,8,7,4,3,1)$, and $I=\{2,4\}$. The lozenges weighted by $\frac{1}{2}$ are denoted by shaded ellipses. In both pictures, the left-most strips (shaded) are not part of the regions.}
    \label{fcb}
\end{figure}

We first define the regions $A(m;\lambda_{st})$ and $\widetilde{A}(m;\lambda_{st})$, which are special cases of the regions $A(m;\lambda_{st};I)$ and $\widetilde{A}(m;\lambda_{st};I)$ when $I=\emptyset$. Let $m$ be a nonnegative integers and $\lambda_{st}=(\lambda_1,\ldots,\lambda_k)$ be a strict partition. We consider $m$ copies of $\wedge$-hook of order $(\lambda_1+1)$ and a copy of shifted $\wedge$-hook of order $\lambda_i$ for $i\in[k]$. Then we concatenate the $m$ copies of $\wedge$-hooks and $n$ copies of shifted $\wedge$-hooks in order from top to bottom along the common axis $\ell$, starting with the largest one in weakly decreasing order. Note that the leftmost strip of the resulting region is tiled in a unique way because of the acute angle at the corner (see the shaded lozenges in the left picture in Figure \ref{fcb}). Since getting rid of the strip from the region will not change the number of tilings, we delete it from the region. Then the resulting region is $\widetilde{A}(m;\lambda_{st})$ and the subregion left to $\ell$ is $A(m;\lambda_{st})$. We also assign weight to lozenges. Every lozenge is weighted by $1$, except the $k$ horizontal lozenges at the right end of $k$ shifted $\wedge$-hook of order $\lambda_i$ for $i\in[k]$. Throughout this paper, lozenges weighted by $\frac{1}{2}$ are denoted by shaded ellipses in all figures. Lastly, every boundary of these two regions is a solid boundary except the boundary of $A(m;\lambda_{st})$ along $\ell$, which is a free boundary. Throughout this paper, free boundaries are denoted by dashed lines. To define $A(m;\lambda_{st};I)$ and $\widetilde{A}(m;\lambda_{st};I)$ for any set $I\subseteq[k]$, we look at the $k$ shifted $\wedge$-hooks of order $\lambda_i$ for $i\in[k]$. For the shifted $\wedge$-hook of order $\lambda_i$, we label the leftmost (left-pointing) unit triangle and rightmost (right-pointing) unit triangle in it by $i$ and $i'$, respectively. Then, for any $I\subseteq[k]$, $A(m;\lambda_{st};I)$ is obtained from $A(m;\lambda_{st})$ by deleting unit triangles labeled by the elements of $I$. Similarly, $\widetilde{A}(m;\lambda_{st};I)$ is obtained from $\widetilde{A}(m;\lambda_{st})$ by deleting unit triangles labeled by $i$ and $i'$ for $i\in I$ (examples are shown in the two pictures in Figure \ref{fcb}). The main theorem of this section is the following.

\begin{thm}\label{tca}
For a nonnegative integer $m$, a strict partition $\lambda_{st}$ with $k$ parts, and a set $I\subseteq[k]$,
\begin{equation}\label{eca}
    \M_f(A(m;\lambda_{st};I))^2=2^{k-|I|}\M(\widetilde{A}(m;\lambda_{st};I)).
\end{equation}
\end{thm}

The proof of Theorems \ref{tca} is organized as follows. We first construct a bijection between lozenge tilings and nonintersecting paths to convert the problem into the nonintersecting paths enumeration problem. Then, we apply Theorem \ref{tac} to finish the proof.

\begin{proof}[Proof of Theorem 3.1]
If $I=[k]$, then \eqref{eca} is true because the equation becomes\footnote{This is because, the two regions appearing in \eqref{eca} both have only one tiling of weight $1$ and $k-|I|=0$.} $1=1$. Thus, we assume that $I\subsetneq[k]$. We now construct lattices from the regions $A(m;\lambda_{st};I)$ and $\widetilde{A}(m;\lambda_{st};I)$.

\begin{figure}
    \centering
    \includegraphics[width=1.0\textwidth]{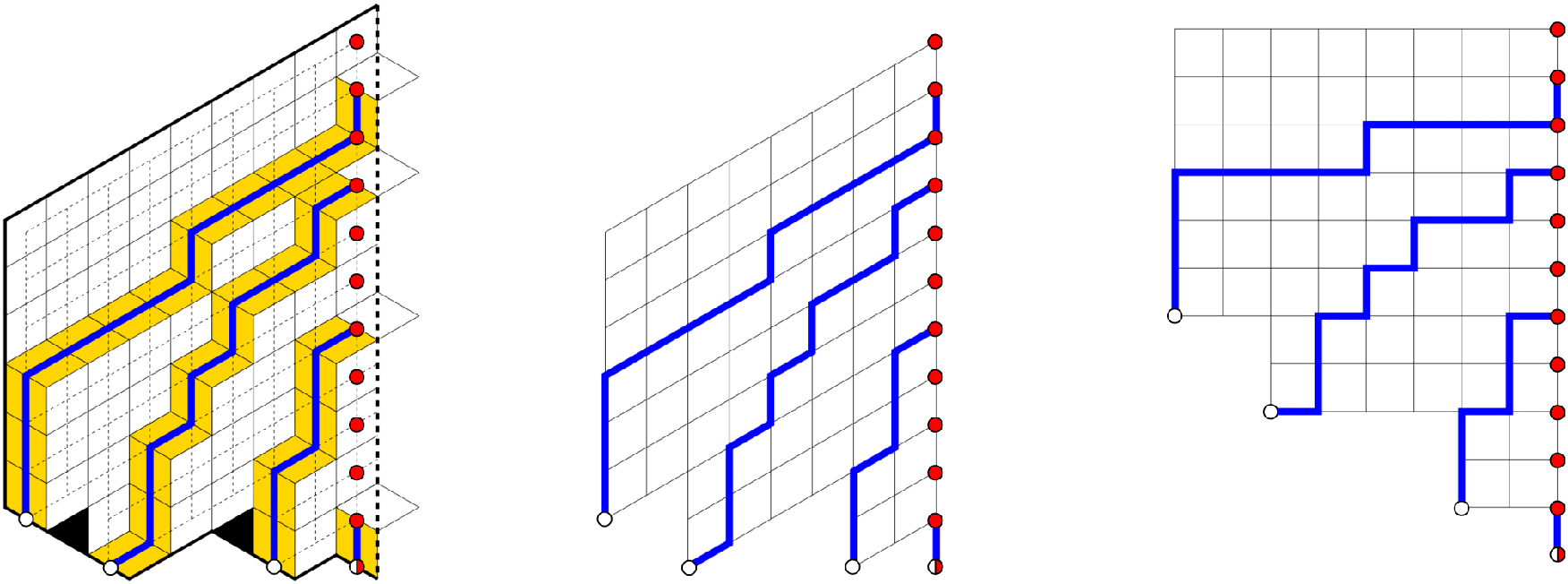}
    \caption{A paths of lozenges from a lozenge tilings of $A(m;\lambda_{st};I)$ in the right picture in Figure 3.2 (left), the lattice constructed from $A(m;\lambda_{st};I)$ and nonintersecting paths on it (middle) and realization of the lattice paths on $\mathbb{Z}^2$ (right). Starting and ending points are marked by empty and filled circles, respectively. The half-filled circle is both the starting and ending points (in this case, $u_6$ and $v_{12}$ coincide).}
    \label{fcc}
\end{figure}

To construct a lattice from $A(m;\lambda_{st};I)$, we mark the midpoints of all unit segments with negative slope in $A(m;\lambda_{st};I)$. In particular, we denote the point on the boundary of the unit triangle with label $i$ by $u_{i}$ for $i\in[k]\setminus I$ and the $(m+k)$ points on the immediate left to $\ell$ by $v_{1},\ldots,v_{m+k}$ from top to bottom. We join two points that are one unit away from each other and orient all vertical unit segments upward and all nonvertical unit segments from bottom-left to top-right to form a lattice. If we assign weight $1$ to every unit segment in the lattice, then we claim that the number of families of $(k-|I|)$ nonintersecting lattice paths with fixed starting points $\textbf{u}_I=(u_i)_{i\in[k]\setminus I}$ and unfixed ending points $\textbf{v}=(v_1,\ldots,v_{m+k})$ equals the number of lozenge tilings of $A(m;\lambda_{st};I)$. This can be checked as follows. Given a lozenge tiling of $A(m;\lambda_{st};I)$, color all the lozenges that do not cross $\ell$ whose long diagonals have either zero or negative slope. Those lozenges form $(k-|I|)$ paths of lozenges (see the left picture in Figure \ref{fcc}). Then, for each point $u_i$, join the point and the midpoint of the opposite side of the lozenge that contains $u_i$. Continue this process following each path of lozenges, then we get $(k-|I|)$ nonintersecting lattice paths from $\textbf{u}_I\coloneq(u_i)_{i\in[k]\setminus I}$ to $(k-|I|)$ points among $\textbf{v}=(v_1,\ldots,v_{m+k})$ (see the middle picture in Figure \ref{fcc}). One can readily check that this correspondence is invertible and thus, we have $\M_{f}(A(m;\lambda_{st};I))=\GF[\mathscr{P}_0(\textbf{u}_I,\textbf{v})]$. Note that our lattice paths can be regarded as lattice paths on $\mathbb{Z}^2$: set $u_i=(-\lambda_i+1,k-i)$ for $i\in[k]\setminus I$ and $v_j=(0,m+k-j)$ for $j\in[m+k]$ with all edges oriented toward the top and right (see the right picture in Figure \ref{fcc}). Then $\mathbf{u}_I$ and $\mathbf{v}$ are compatible and the path matrix $\M(\mathbf{u}_I,\mathbf{v})$ is the $(k-|I|)\times(m+k)$ matrix
\begin{equation}\label{ecb}
    \M(\mathbf{u}_I,\mathbf{v})=\Bigg[\binom{(0-(-\lambda_i+1))+((m+k-j)-(k-i))}{((m+k-j)-(k-i))}\Bigg]=\Bigg[\binom{\lambda_i-1+m+i-j}{m+i-j}\Bigg],
\end{equation}
whose rows are indexed by $i\in[k]\setminus I$ and columns are indexed by $j
\in[m+k]$. Thus, from Theorem \ref{tac}, we have
\begin{equation}\label{ecc}
    \M_{f}(A(m;\lambda_{st};I))^2=\GF[\mathscr{P}_0(\textbf{u}_I,\textbf{v})]^2=\det\Big[\M(\mathbf{u}_I,\mathbf{v})U_{m+k}\M(\mathbf{u}_I,\mathbf{v})^T\Big].
\end{equation}

\begin{figure}
    \centering
    \includegraphics[width=0.68\textwidth]{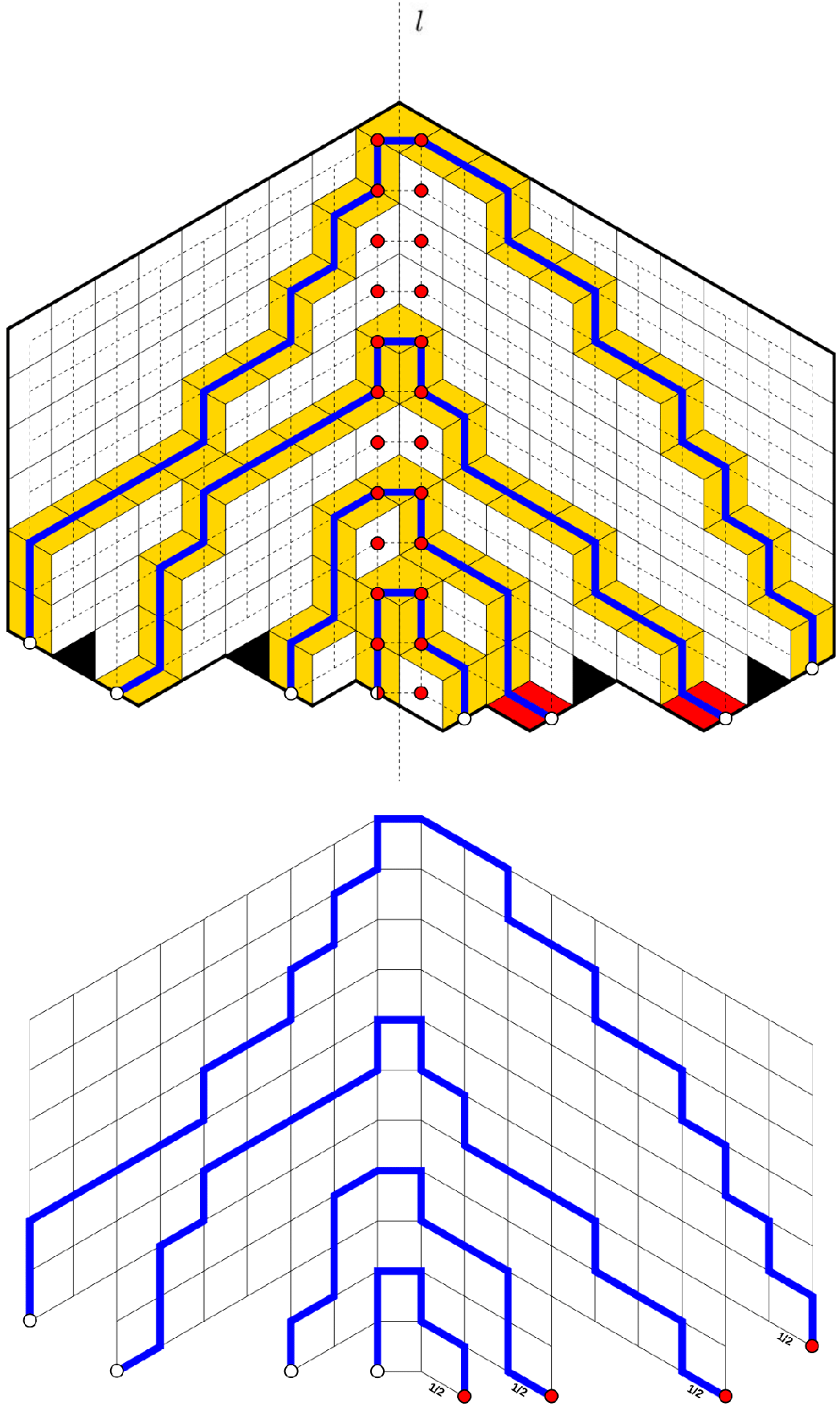}
    \caption{A path of lozenges obtained from a lozenge tiling of $\widetilde{A}(m;\lambda_{st};I)$ in the left picture in Figure \ref{fcb} (top) and the lattice constructed from $\widetilde{A}(m;\lambda_{st};I)$ and nonintersecting paths on it (bottom). In the top picture, the lozenges weighted by $\frac{1}{2}$ are colored red (darker), while all other lozenges in the paths of lozenges are colored gold (brighter). In the bottom picture, the edges weighted by $\frac{1}{2}$ are marked with $1/2$.}
    \label{fcd}
\end{figure}

To construct a lattice from $\widetilde{A}(m;\lambda_{st};I)$, we need three steps.
\begin{enumerate}
    \item On the subregion left to $\ell$, we follow the construction as we did for $A(m;\lambda_{st};I)$.
    \item On the subregion right to $\ell$, we mark the midpoints of all unit segments in it with a ``positive" slope. We denote the $(m+k)$ points on the immediate right to $\ell$ by $v_1',\ldots,v_{m+k}'$ from top to bottom and the points on the unit triangle with label $i'$ by $u_{i}'$ for $i\in[k]\setminus I$. We join every pair of vertices one unit away from each other by a unit segment. Unlike the previous step, we orient all vertical unit segments ``downward" and all nonvertical unit segments ``from top-left to bottom-right". Every unit segment is weighted by $1$, except the unit segment joining $u_{i}'$ and the point on the top-left of it, which is weighted by $\frac{1}{2}$, for $i\in[k]\setminus I$.
    \item We join $v_i$ and $v_i'$ for $i\in[m+k]$. These unit segments have unit weight and are oriented to the right.
\end{enumerate}
We now claim that there is a weight-preserving bijection between lozenge tilings of $\widetilde{A}(m;\lambda_{st};I)$ and families of $(k-|I|)$ nonintersecting lattice paths with fixed starting points $\textbf{u}_I=(u_i)_{i\in[k]\setminus I}$ and fixed ending points $\textbf{u}'_I=(u_i')_{i\in[k]\setminus I}$ on the lattice describe above. To see this, given a lozenge tiling of $\widetilde{A}(m;\lambda_{st};I)$, we color
\begin{enumerate}
    \item all the lozenges left to $\ell$ whose long diagonals have either negative slope or zero slope,
    \item all the lozenges right to $\ell$ whose long diagonals have either positive slope or zero slope, and
    \item all the horizontal lozenges crossing $\ell$.
\end{enumerate}
This time, we obtain $(k-|I|)$ nonintersecting paths of lozenges (see the picture on top in Figure \ref{fcd}). Then, following the lattice paths construction from lozenge tilings of $A(m;\lambda_{st};I)$, we get $(k-|I|)$ nonintersecting paths between $\textbf{u}_I$ and $\textbf{u}'_I$ on the lattice described above (see the bottom picture in Figure \ref{fcd}). Again, this correspondence is bijective and also weight-preserving. Furthermore, since $\textbf{u}_I$ and $\textbf{u}'_I$ are compatible, by Lindstr\"{o}m--Gessel--Viennot theorem, we have
\begin{equation}\label{ecd}
    \M(\widetilde{A}(m;\lambda_{st};I))=\GF[\mathscr{P}_0(\textbf{u}_I,\textbf{u}'_I)]=\det[\M(\mathbf{u}_I,\mathbf{u}'_I)].
\end{equation}
By the construction of the lattice, the path matrix $\M(\mathbf{u}_I,\mathbf{u}'_I)$ satisfies
\begin{equation}\label{ece}
    \M(\mathbf{u}_I,\mathbf{u}'_I)=\M(\mathbf{u}_I,\mathbf{v})\M(\mathbf{v},\mathbf{v}')\M(\mathbf{v}',\mathbf{u}'_I)
\end{equation}
and
\begin{equation}\label{ecf}
    \M(\mathbf{v},\mathbf{v}')=I_{m+k}
\end{equation}
is the identity matrix.
Furthermore, using a similar embedding on $\mathbb{Z}^2$ and the fact that the set of lattice paths from a vertex in $\mathbf{v}'$ and a vertex in $\mathbf{u}'_I$ can be partitioned according to the last step, one can find the $(m+k)\times(k-|I|)$ path matrix $\M(\mathbf{v}',\mathbf{u}'_I)$ as follows:
\begin{equation}\label{ecg}
    \M(\mathbf{v}',\mathbf{u}'_I)=\Bigg[\frac{1}{2}\binom{\lambda_j-1+m+j-i}{m+j-i}+\binom{\lambda_j-1+m+j-i}{m+j-i-1}\Bigg],
\end{equation}
where rows are indexed by $i\in[m+k]$ and columns are indexed by $j\in[k]\setminus I$. Since\footnote{In the first equality in \eqref{ech}, we use the well-known identity $\sum_{i=0}^k\binom{n+i}{i}=\binom{n+k+1}{k}$ for $k,n\in\mathbb{Z}_{\geq0}$. In the second equality in \eqref{ech}, we use the usual convention that $\binom{n}{k}=0$ if $k<0$.}
\begin{equation}\label{ech}
\begin{aligned}
    &\frac{1}{2}\binom{\lambda_j-1+m+j-i}{m+j-i}+\binom{\lambda_j-1+m+j-i}{m+j-i-1}\\
    =&\frac{1}{2}\binom{\lambda_j-1+m+j-i}{m+j-i}+\sum_{k=i+1}^{m+j}\binom{\lambda_j-1+m+j-k}{m+j-k}\\
    =&\frac{1}{2}\binom{\lambda_j-1+m+j-i}{m+j-i}+\sum_{k=i+1}^{m+k}\binom{\lambda_j-1+m+j-k}{m+j-k},
\end{aligned}
\end{equation}
one can deduce that
\begin{equation}\label{eci}
    \M(\mathbf{v}',\mathbf{u}'_I)=\bigg(\frac{1}{2}U_{m+k}\bigg)\M(\mathbf{u}_I,\mathbf{v})^{T},
\end{equation}
where $U_{m+k}$ is the $(m+k)\times(m+k)$ upper triangular matrix defined in the statement of Theorem \ref{tac}. Combining \eqref{ecd}--\eqref{ecf}, \eqref{eci}, and \eqref{ecc}, we have
\begin{equation}\label{ecj}
\begin{aligned}
    \M(\widetilde{A}(m;\lambda_{st};I))=\det\big[\M(\mathbf{u}_I,\mathbf{v})\M(\mathbf{v},\mathbf{v}')\M(\mathbf{v}',\mathbf{u}'_I)\big]   =&\det\bigg[\M(\mathbf{u}_I,\mathbf{v})\bigg(\frac{1}{2}U_{m+k}\bigg)\M(\mathbf{u}_I,\mathbf{v})^{T}\bigg]\\
    =&\det\bigg[\frac{1}{2}\Big(\M(\mathbf{u}_I,\mathbf{v})U_{m+k}\M(\mathbf{u}_I,\mathbf{v})^{T}\Big)\bigg]\\
    =&2^{-(k-|I|)}\det\Big[\M(\mathbf{u}_I,\mathbf{v})U_{m+k}\M(\mathbf{u}_I,\mathbf{v})^{T}\Big]\\
    =&2^{-(k-|I|)}\M_{f}(A(m;\lambda_{st};I))^2. 
\end{aligned}
\end{equation}
Multiplying both sides of \eqref{ecj} by $2^{k-|I|}$, we have
\begin{equation}\label{eck}
    \M_{f}(A(m;\lambda_{st};I))^2=2^{k-|I|}\M(\widetilde{A}(m;\lambda_{st};I))
\end{equation}
and this completes the proof.
\end{proof}

As mentioned in the introduction, we show several applications of Theorem \ref{tca} and one more application of Theorem \ref{tac} in the next section.

\section{Applications of Theorems \ref{tca} and \ref{tac}}\label{section4}

We present three applications of Theorems \ref{tca} and \ref{tac}. The three applications are as follows.
\begin{enumerate}
    \item Find a new family of regions whose tiling generating function is given by a simple product formula (see Theorem \ref{tda}).
    \item Give a simpler proof of a factorization theorem for lozenge tilings of hexagons with holes (see Theorem \ref{tdb}).
    \item Provide new determinant formulas for the volume generating functions of shifted plane partitions of a shifted shape and symmetric plane partitions of a symmetric shape (see Theorem \ref{tdd}).
\end{enumerate}

For nonnegative integers $n$ and $k$ such that $0\leq k\leq n$, consider a strict partition $\lambda_{st}^{n,k}\coloneqq(n,\ldots,1)+(k,\ldots,1)$. This is a partition whose $i$-th entry is $(n+k+2-2i)$ for $i=1,\ldots,k$ and $(n+1-i)$ for $i=k+1,\ldots,n$. We then consider the two regions $A(m;\lambda_{st}^{n,k})$ and  $\widetilde{A}(m;\lambda_{st}^{n,k})$ for nonnegative integer $m$ (see the two pictures in Figure \ref{fda}). Our first application of Theorem \ref{tca} is a simple product formula for the tiling generating function of $\widetilde{A}(m;\lambda_{st}^{n,k})$.

\begin{figure}
    \centering
    \includegraphics[width=1.0\textwidth]{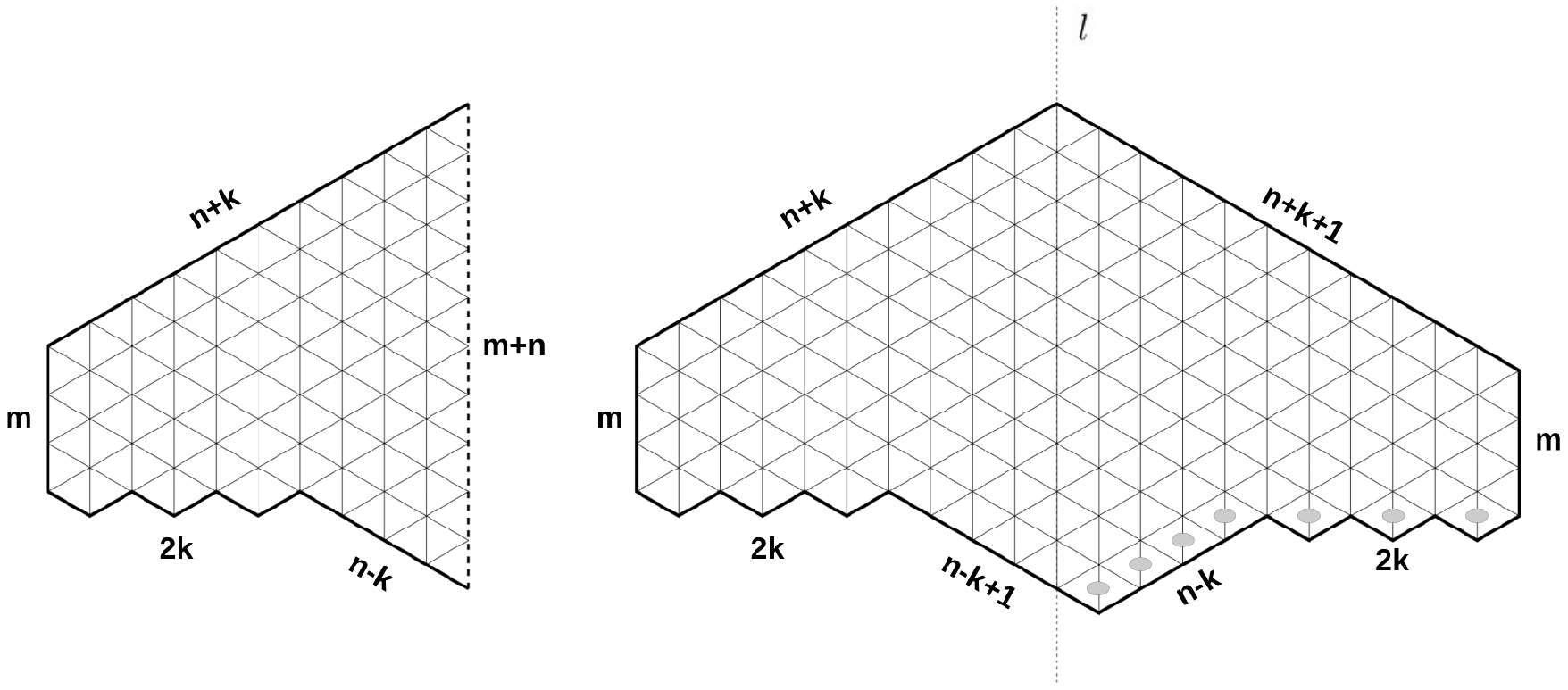}
    \caption{The regions $A(m;\lambda_{st}^{n,k})$ (left) and  $\widetilde{A}(m;\lambda_{st}^{n,k})$ (right), where $n=7$, $m=3$, and $k=3$.}
    \label{fda}
\end{figure}

\begin{thm}\label{tda}
    For nonnegative integers $m$, $n$, $k$ such that $0\leq k\leq n$ and a strict partition $\lambda_{st}^{n,k}=(n,\ldots,1)+(k,\ldots,1)$, the tiling generating function of $\widetilde{A}(m;\lambda_{st}^{n,k})$ is given by the following simple product formula.
    \begin{equation}\label{eda}
        \M(\widetilde{A}(m;\lambda_{st}^{n,k}))=\frac{1}{2^{n}}\Bigg[\prod_{i\leq i\leq j\leq n}\frac{m+i+j-1}{i+j-1}\prod_{1\leq i\leq j\leq i}\frac{m+i+j}{i+j}\Bigg]^2.
    \end{equation}
\end{thm}
\begin{proof}
The proof is based on the simple product formula for the number of lozenge tilings of $A(m;\lambda_{st}^{n,k})$, which is proved independently by Hopkins and Lai \cite{HL} and Okada \cite{O2} in two different ways. They explained that the set of lozenge tilings of $A(m;\lambda_{st}^{n,k})$ is in bijection with the set of shifted plane partitions of shifted double staircase shape $\lambda_{st}^{n,k}$ with largest entry no greater than $m$ and showed\footnote{In fact, they proved more general results. \eqref{edb} is a special case of their results.} that $\M_f(A(m;\lambda_{st}^{n,k}))$ is given by the following formula.
\begin{equation}\label{edb}
    \M_f(A(m;\lambda_{st}^{n,k}))=\prod_{i\leq i\leq j\leq n}\frac{m+i+j-1}{i+j-1}\prod_{1\leq i\leq j\leq i}\frac{m+i+j}{i+j}.
\end{equation}
Combining with Theorem \ref{tca}, the proof is completed.
\end{proof}
It would be interesting if one could prove \eqref{eda} directly without using \eqref{edb} and Theorem \ref{tca}. If one can find such a proof, we can deduce \eqref{edb} using Theorem \ref{tca}. This will give a new proof of \eqref{edb}.

Next, we give a simple proof of a factorization theorem for lozenge tilings of hexagons with holes, which was first proved by Ciucu \cite{C5}. To state the theorem, we define some hexagonal regions with holes. First, for positive integers $m$ and $n$, we consider a hexagon with sides of length $2m$, $n$, $n$, $2m$, $n$, and $n$ clockwise from the left side. Denote the resulting hexagon by $H_{2m,n}$. This hexagon has $n$ horizontal lozenges along its horizontal symmetry axis. When $n$ is even, we label them by $1,\ldots,\frac{n}{2},\frac{n}{2}',\ldots,1'$ from left to right, while we label them by $1,\ldots,\frac{n-1}{2},\frac{n+1}{2},\frac{n-1}{2}',\ldots,1'$ from left to right if $n$ is odd. Then, for any set of positive integers $K=\{k_1,\ldots,k_s\}$ such that $1\leq k_1<\ldots<k_s\leq\frac{n}{2}$, the region $H_{2m,n}(K)$ is obtained from $H_{2m,n}$ by deleting the left-pointing triangles of size $2$ that contain horizontal lozenges with labels $k_1,\ldots,k_s$ and the right-pointing triangles of size $2$ that contain horizontal lozenges with labels $k_1',\ldots,k_s'$ (see the left picture in Figure \ref{fdb}). Ciucu considered one more family of region $H_{2m,2n-1}(K;2x-1)$ for positive integers $m$, $n$, $x$ such that $x\leq n$ and a set of positive integers $K=\{k_1,\ldots,k_s\}$ such that $1\leq k_1<\ldots<k_s\leq\ (n-x)$. To construct this region, we first consider $H_{2m,2n-1}$ and remove a horizontal lozenge of size $(2x-1)$ from its center (we call it a ``punctured hexagon"). The resulting region is symmetric across its horizontal symmetry axis, and there are $(2n-2x)$ horizontal lozenges along the axis. We label the horizontal lozenges by $1,\ldots,(n-x),(n-x)',\ldots,1'$ from left to right. Then, $H_{2m,2n-1}(K;2x-1)$ is obtained from the punctured hexagon by deleting the left-pointing triangles of size $2$ that contain horizontal lozenges with labels $k_1,\ldots,k_s$ and the right-pointing triangles of size $2$ that contain horizontal lozenges with labels $k_1',\ldots,k_s'$ (see the right picture in Figure \ref{fdb}).

\begin{figure}
    \centering
    \includegraphics[width=1.0\textwidth]{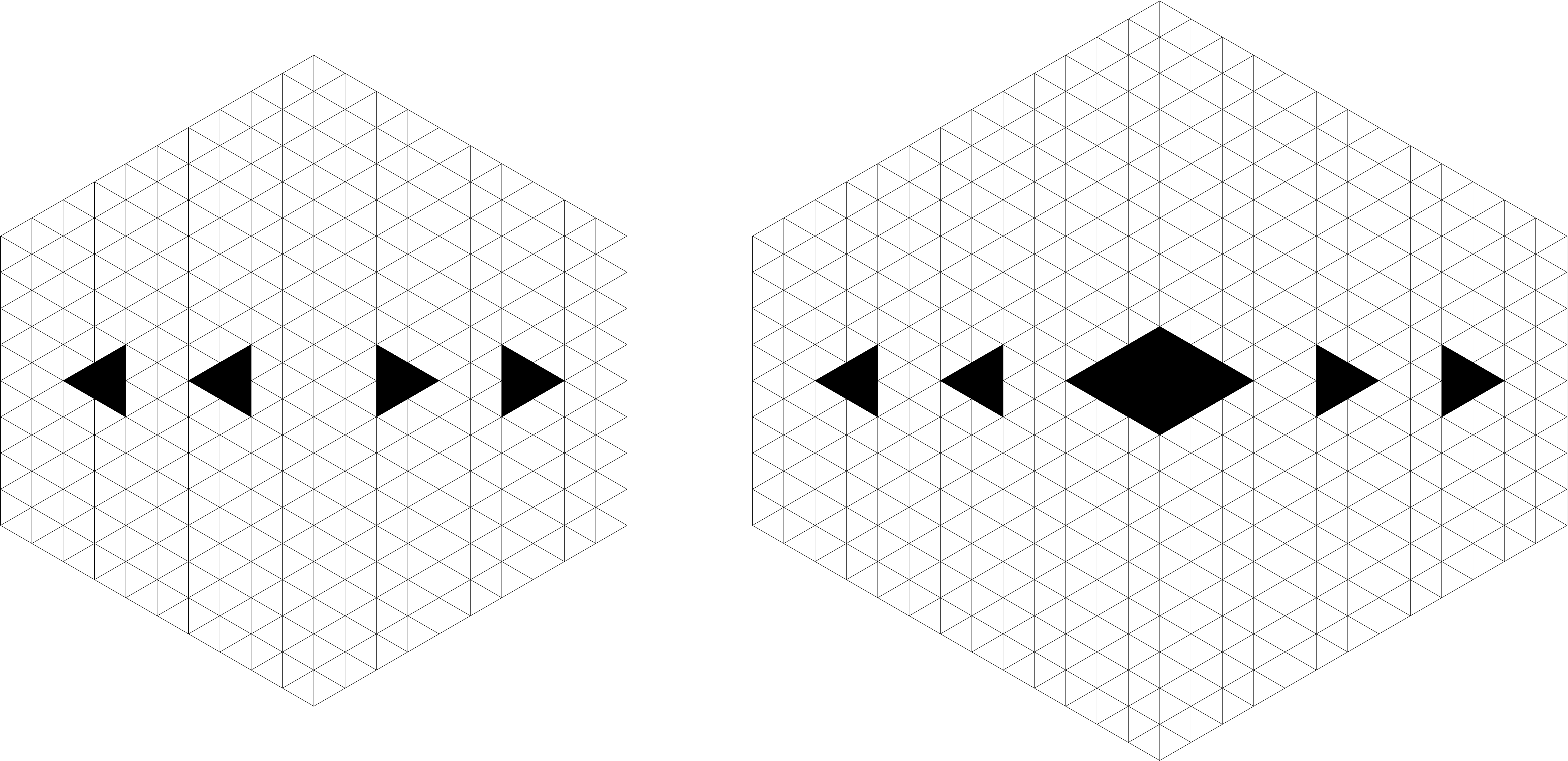}
    \caption{$H_{8,10}(\{2,4\})$(left) and $H_{8,13}(\{2,4\},3)$ (right).}
    \label{fdb}
\end{figure}

For any region $R$ on a triangular lattice, $R$ is \textit{centrally symmetric} if it is invariant under rotation by $180^{\circ}$. When $R$ is centrally symmetric, a lozenge tiling of $R$ is \textit{centrally symmetric} if the lozenge tiling is invariant under rotation by $180^{\circ}$ and the number of centrally symmetric lozenge tilings of $R$ is denoted by $\M_{\odot}(R)$. Similarly, a region $R$ is \textit{vertically symmetric} if it is invariant under reflection across its vertical symmetry axis, and its lozenge tiling is \textit{vertically symmetric} if the lozenge tiling is invariant under reflection across the vertical symmetry axis. The number of vertically symmetric lozenge tilings of $R$ is denoted by $\M_{|}(R)$. When a region $R$ is both centrally symmetric and vertically symmetric, $\M_{\odot,|}(R)$ denotes the number of lozenge tilings of $R$ that are both centrally symmetric and vertically symmetric. Note that the two families of regions $H_{2m,n}(K)$ and $H_{2m,2n-1}(K;2x-1)$ are both centrally symmetric and vertically symmetric. In \cite{C5}, Ciucu showed that if $R$ is $H_{2m,n}(K)$ or $H_{2m,2n-1}(K;2x-1)$, then $\M_{\odot}(R)$ and $\M_{\odot,|}(R)$ has the following simple relation.

\begin{thm}[Theorem 2.1 in \cite{C5}]\label{tdb}
Let $m$, $n$, and $x$ be positive integers such that $x\leq n$.

(a) For any set of positive integers $K=\{k_1,\ldots,k_s\}$ such that $1\leq k_1\leq\ldots\leq k_s\leq\frac{n}{2}$,
\begin{equation}\label{edc}
    \M_{\odot}(H_{2m,n}(K))=\M_{\odot,|}(H_{2m,n}(K))^2.
\end{equation}

(b) For any set of positive integers $K=\{k_1,\ldots,k_s\}$ such that $1\leq k_1\leq\ldots\leq k_s\leq\ (n-x)$,
\begin{equation}\label{edd}
    \M_{\odot}(H_{2m,2n-1}(K;2x-1))=\M_{\odot,|}(H_{2m,2n-1}(K;2x-1))^2.
\end{equation}
\end{thm}

Ciucu pointed out that when $K=\emptyset$ in Theorem \ref{tdb} (a), via the bijection of David and Tomei (see \cite{DT}), it becomes the factorization formula that ``self-complementary" plane partitions and ``symmetric and self-complementary" plane partitions satisfy, which follows from the product formulas for these two symmetry classes by Stanley \cite{St} and Proctor \cite{P}. For the readers who are interested in the enumeration of symmetry classes of plane partitions, we refer the readers to \cite{St}, \cite{Kup}, and \cite{Kra3}.

Ciucu's proof of the above theorem consists of two parts. 
\begin{enumerate}
    \item First, he showed that for $R=H_{2m,n}(K)$ or $H_{2m,2n-1}(K;2x-1)$, $\M_{\odot}(R)$ and $\M_{\odot,|}(R)$ can be reduced to enumeration of lozenge tilings of two subregions of $R$.
    \item Then, he enumerated lozenge tilings of the two subregions using his previous work in \cite{C3} and \cite{C4} and checked that both sides of \eqref{edc} and \eqref{edd} match up.
\end{enumerate}

We now give a new proof of Theorem \ref{tdb}. We will follow\footnote{When we follow the first part of Ciucu's proof in \cite{C5}, instead of giving a full detail, we give a brief idea of Ciucu's argument and refer the reader to \cite{C5}.} the first part of Ciucu's proof. Then, instead of evaluating both sides of \eqref{edc} and \eqref{edd}, we show that they should match up because of Theorem \ref{tca}. Our new proof answers Ciucu's question about finding a direct proof of his theorem without explicitly evaluating both sides of the equations.

\begin{proof}[Proof of Theorem \ref{tdb}]

\begin{figure}
    \centering
    \includegraphics[width=0.9\textwidth]{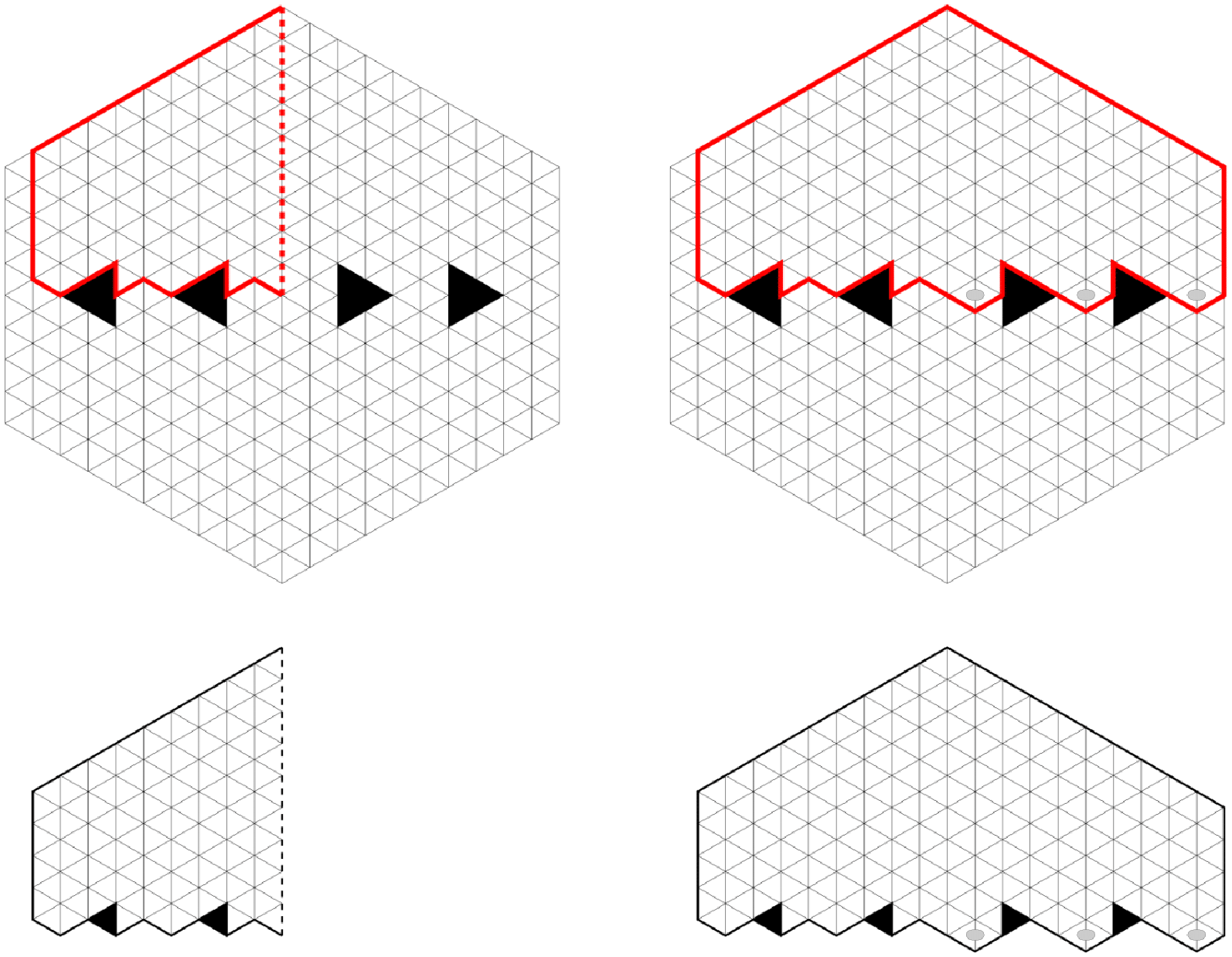}
    \caption{The regions $Q_{4,10}(\{2,4\})$ (top-left, bounded by thick lines) and $A(4;(9,7,\ldots,1);\{2,4\})$ (bottom-left) are the same regions, and thus, have the same number of lozenge tilings. Similarly, The regions $R_{4,10}(\{2,4\})$ (top-right, bounded by thick lines) and $\widetilde{A}(4;(9,7,\ldots,1);\{2,4\})$ (bottom-right) are the same regions, and thus, have the same tiling generating functions.}
    \label{fdc}
\end{figure}

We first prove part (a) for $n$ even. In \cite{C5}, Ciucu showed that the number of both centrally symmetric and vertically symmetric lozenge tilings of $H_{2m,n}(K)$ is the same as the number of lozenge tilings of roughly a quarter of the region that has a free boundary on the right side, as shown in the top-left picture in Figure \ref{fdc}. This is because a tiling is both centrally symmetric and vertically symmetric if and only if it is both vertically symmetric and horizontally symmetric\footnote{A region is \textit{horizontally symmetric} if it is invariant under reflection across its horizontal symmetry axis. A tiling of a horizontally symmetric region is \textit{horizontally symmetric} if the tiling is invariant under reflection across its horizontal symmetry axis.}. Thus, every horizontal lozenge on the horizontal symmetry axis should be a part of a tiling, and the tiling is uniquely determined by the tiling of its top-left subregion. The free boundary is due to the vertical symmetry condition. Let us denote this region by $Q_{m,n}(K)$ (it is the subregion enclosed by thick lines in the top-left picture in Figure \ref{fdc}). Therefore, we have
\begin{equation}\label{ede}
    \M_{\odot,|}(H_{2m,n}(K))=\M_f(Q_{m,n}(K)).
\end{equation}
See \cite{C5} for more details about how \eqref{ede} is derived. One important observation is, this region $Q_{m,n}(K)$ has the same number of lozenge tilings as $A(m;(n-1,n-3,\ldots,1); K)$ because they are the same regions (compare the two pictures on the left in Figure \ref{fdc}). Thus, we have
\begin{equation}\label{edf}
    \M_f(Q_{m,n}(K))=\M_f(A(m;(n-1,n-3,\ldots,1); K)).
\end{equation}

In case of centrally symmetric lozenge tilings of $H_{2m,n}(K)$, using an extension of his Matching Factorization Theorem (see the proof of Theorem 7.1 in \cite{C1} or Theorem 6.1 in \cite{C7}) to the quotient graph of the dual graph of $H_{2m,n}(K)$ with respect to the rotation by $180^{\circ}$, he showed that the number of centrally symmetric lozenge tilings of $H_{2m,n}(K)$ is equal to $2^{n/2-s}$ times the tiling generating function of roughly a half of the region. More precisely, cut the region $H_{2m,n}(K)$ along the top sides of lozenges labeled by $i$ for $i\in[\frac{n}{2}]\setminus K$ and the bottom sides of lozenges labeled by $i'$ for $i\in[\frac{n}{2}]\setminus K$. Then the region is split into two isomorphic subregions. Consider the one above and delete the left-most strip, which is uniquely tiled due to an acute angle at the corner. We denote the resulting region by $R_{m,n}(K)$ (see the top-right picture in Figure \ref{fdc} for an example). Every lozenges in $R_{m,n}(K)$ are weighted by $1$, except the horizontal lozenges labeled by $i'$ for $i\in[\frac{n}{2}]\setminus K$: they are weighted by $\frac{1}{2}$. Then Ciucu showed
\begin{equation}\label{edg}
    \M_{\odot}(H_{2m,n}(K))=2^{n/2-s}\M(R_{m,n}(K)).
\end{equation}
See \cite{C5} for more details about how \eqref{edg} is derived. Another important observation is, the region $R_{m,n}(K)$ has the same tiling generating function as $\widetilde{A}(m;(n-1,n-3,\ldots,1); K)$. This is again because they are the same regions (see the two pictures on the right in Figure \ref{fdc}). Hence, we have
\begin{equation}\label{edh}
    \M_{\odot}(R_{m,n}(K))=\M(\widetilde{A}(m;(n-1,n-3,\ldots,1); K)).
\end{equation}
Thus, to prove \eqref{edc}, it suffices to show
\begin{equation}\label{edi}
    2^{n/2-s}\M(\widetilde{A}(m;(n-1,n-3,\ldots,1); K))=\M(A(m;(n-1,n-3,\ldots,1); K))^2.
\end{equation}
This follows from Theorem \ref{tca} and completes the proof for $n$ even.

\begin{figure}
    \centering
    \includegraphics[width=1.0\textwidth]{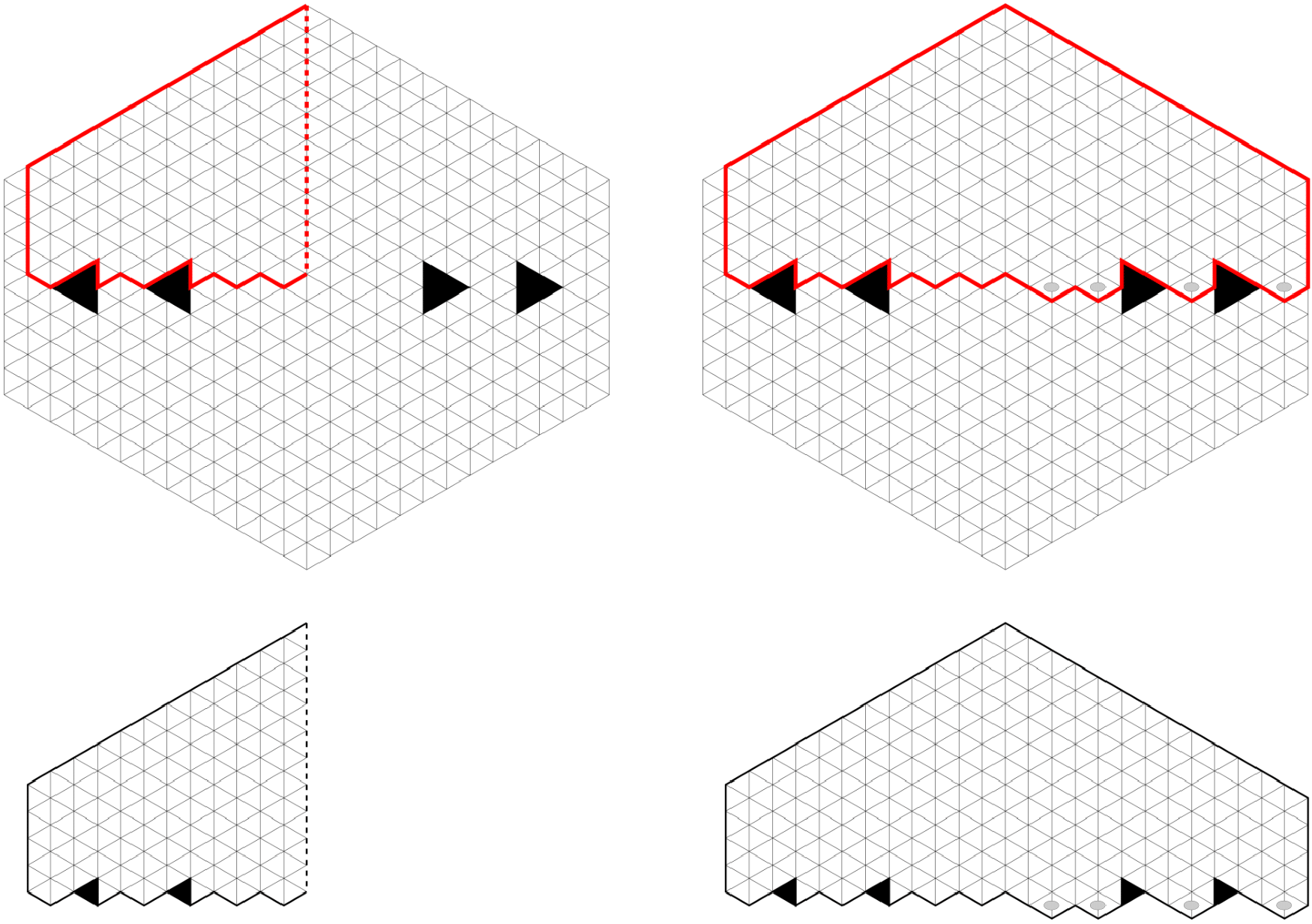}
    \caption{In \cite{C5}, Ciucu showed that the number of lozenge tilings (or the tiling generating function) of the subregion bounded by the thick lines on the top-left (or top-right) picture determines the number of centrally symmetric and vertically symmetric lozenge tilings (or centrally symmetric lozenge tilings) of $H_{8,13}(\{2,4\})$. The subregion on the top-left is the same as the region on the bottom-left, which is $A(4;(12,10,\ldots,2);\{2,4\})$. Similarly, the subregion on the top-right is the same as the region on the bottom-right, which is $\widetilde{A}(4;(12,10,\ldots,2);\{2,4\})$.}
    \label{fdd}
\end{figure}

For $n$ odd, the proof is almost the same as when $n$ is even. Following the same idea, one can show (see Figure \ref{fdd}. We again refer the reader to \cite{C5} for more details of why the equations \eqref{edj} and \eqref{edk} hold.)
\begin{equation}\label{edj}
    \M_{\odot,|}(H_{2m,n}(K))=\M_f(A(m;(n-1,n-3,\ldots,2); K)).
\end{equation}
and
\begin{equation}\label{edk}
    \M_{\odot}(H_{2m,n}(K))=2^{(n-1)/2-s}\M(\widetilde{A}(m;(n-1,n-3,\ldots,2); K)).
\end{equation}
Then from Theorem \ref{tca}, we have
\begin{equation}\label{edl}
\begin{aligned}
    \M_{\odot}(H_{2m,n}(K))&=2^{(n-1)/2-s}\M(\widetilde{A}(m;(n-1,n-3,\ldots,2); K))\\
    &=\M(A(m;(n-1,n-3,\ldots,2); K))^2\\
    &=\M_{\odot,|}(H_{2m,n}(K))^2.    
\end{aligned}
\end{equation}
and this completes the proof for $n$ odd.

\begin{figure}
    \centering
    \includegraphics[width=1.0\textwidth]{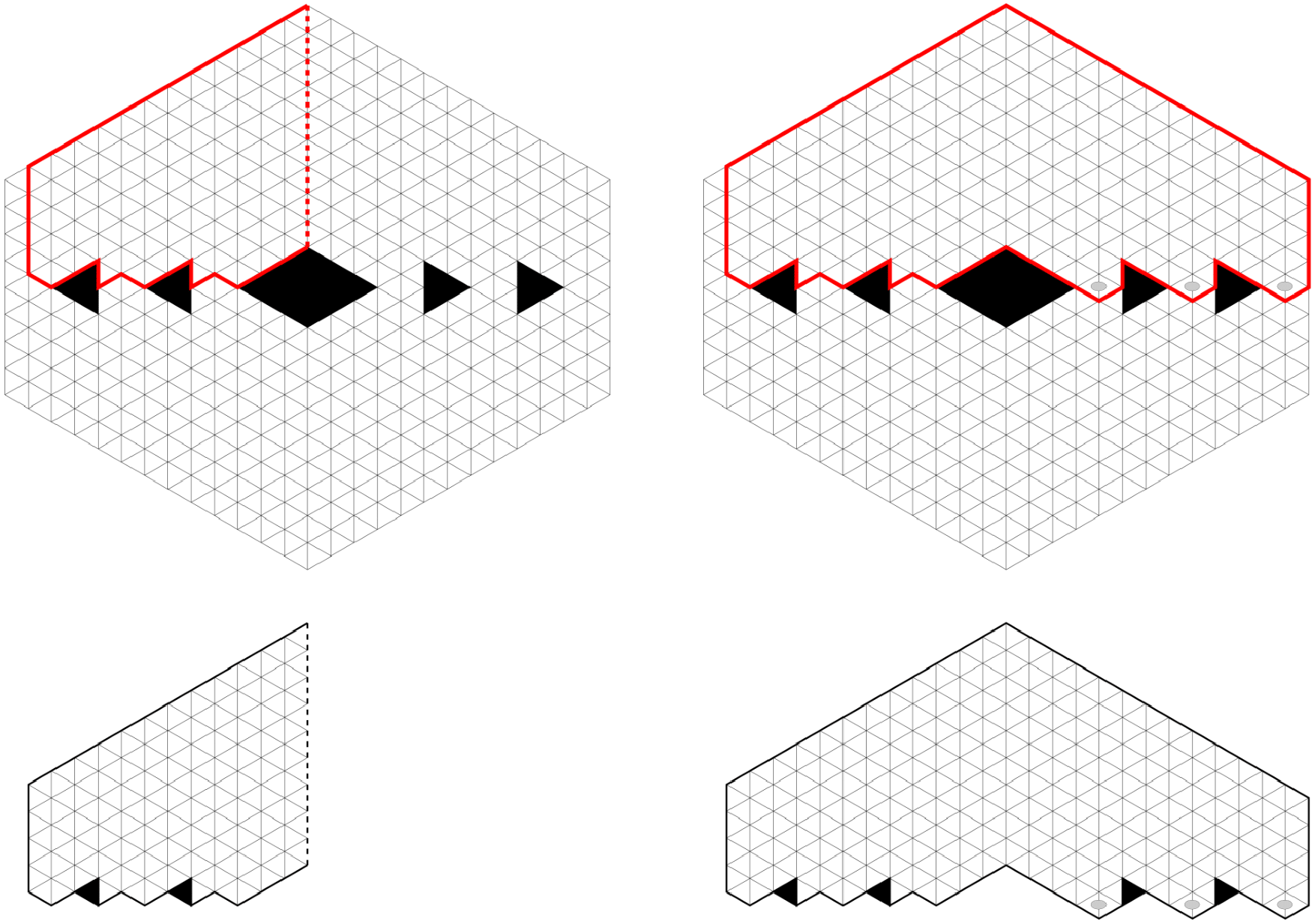}
    \caption{In \cite{C5}, Ciucu showed that the number of lozenge tilings (or the tiling generating function) of the subregion bounded by the thick lines on the top-left (or top-right) picture determines the number of centrally symmetric and vertically symmetric lozenge tilings (or centrally symmetric lozenge tilings) of $H_{8,13}(\{2,4\};3)$. The subregion on the top-left is the same as the region on the bottom-left, which is $A(4;(12,10,\ldots,4);\{2,4\})$. Similarly, the subregion on the top-right is the same as the region on the bottom-right, which is $\widetilde{A}(4;(12,10,\ldots,4);\{2,4\})$.}
    \label{fde}
\end{figure}

The proof of part (b) is almost the same as that of part (a). One can deduce from Ciucu's first part of the proof in \cite{C4} that (see Figure \ref{fde}. The details of why \eqref{tdm} and \eqref{tdn} hold can be found in \cite{C5}.)
\begin{equation}\label{tdm}
    \M_{\odot,|}(H_{2m,2n-1}(K;2x-1))=\M(A(m;(2n-2,2n-4,\ldots,2x); K))
\end{equation}
and
\begin{equation}\label{tdn}
    \M_{\odot}(H_{2m,2n-1}(K;2x-1))=2^{n-x-s}\M(\widetilde{A}(m;(2n-2,2n-4,\ldots,2x); K)).
\end{equation}
Then from Theorem \ref{tca},
\begin{equation}\label{tdo}
\begin{aligned}
    \M_{\odot}(H_{2m,2n-1}(K;2x-1))&=2^{n-x-s}\M(\widetilde{A}(m;(2n-2,2n-4,\ldots,2x); K))\\
    &=\M(A(m;(2n-2,2n-4,\ldots,2x); K))^2\\
    &=\M_{\odot,|}(H_{2m,2n-1}(K;2x-1))^2.
\end{aligned}
\end{equation}
This completes the proof of part (b).
\end{proof}

The last part concerns symmetric plane partitions of a symmetric shape, shifted plane partitions of a shifted shape, and their volume generating functions. Recall that a partition $\lambda=(\lambda_1,\ldots,\lambda_k)$ is a $k$-tuple of positive integers $\lambda_1,\ldots,\lambda_k$ such that $\lambda_1\geq\ldots\geq\lambda_k$. A \textit{Young diagram of the partition $\lambda$} is a collection of left-justified boxes with $\lambda_i$ boxes in the $i$-th row (from the top). A partition is \textit{symmetric} if its Young diagram is invariant under reflection across the main diagonal. A \textit{plane partition of shape $\lambda$} is an array of nonnegative integers $\pi=(\pi_{i,j})_{1\leq i\leq k, 1\leq j\leq\lambda_{i}}$ such that 
$\pi_{i,j}\geq\pi_{i+1,j}$ and $\pi_{i,j}\geq\pi_{i,j+1}$. It can be thought of as writing $\pi_{i,j}$ on the $(i,j)$-position of the Young diagram of shape $\lambda$. A plane partition can be realized as a stack of unit cubes: this can be done by putting $\pi_{i,j}$ unit cubes on the $(i,j)$-position of the $\lambda$'s Young diagram. A \textit{volume} of the plane partition $\pi$ is $|\pi|\coloneqq\sum_{i=1}^{k}\sum_{j=1}^{\lambda_{i}}\pi_{i,j}$, the sum of all entries of $\pi$, which can be thought as the total number of unit cubes in the stack. For any nonnegative integer $m$ and a partition $\lambda$, $PP(m;\lambda)$ is the set of plane partitions of shape $\lambda$ whose largest entry $\pi_{1,1}$ is no greater than $m$. The cardinality of $PP(m;\lambda)$ is denoted by $|PP(m;\lambda)|$ and the volume generating function of $PP(m;\lambda)$ is $|PP(m;\lambda;q)|\coloneqq\sum_{\pi\in PP(m;\lambda)}q^{|\pi|}$. For a symmetric partition $\lambda$, a plane partition $\pi$ of shape $\lambda$ is \textit{symmetric} if it is invariant under reflection across the main diagonal, i.e., $\pi_{i,j}=\pi_{j,i}$ for all $i$ and $j$. The set of symmetric plane partitions of symmetric shape $\lambda$ with largest entry no greater than $m$ is denoted by $PP_{sym}(m;\lambda)$. The cardinality and the volume generating function of $PP_{sym}(m;\lambda)$ are denoted by $|PP_{sym}(m;\lambda)|$ and $|PP_{sym}(m;\lambda;q)|\coloneqq\sum_{\pi\in PP_{sym}(m;\lambda)}q^{|\pi|}$, respectively.

\begin{figure}
    \centering
    \includegraphics[width=0.6\textwidth]{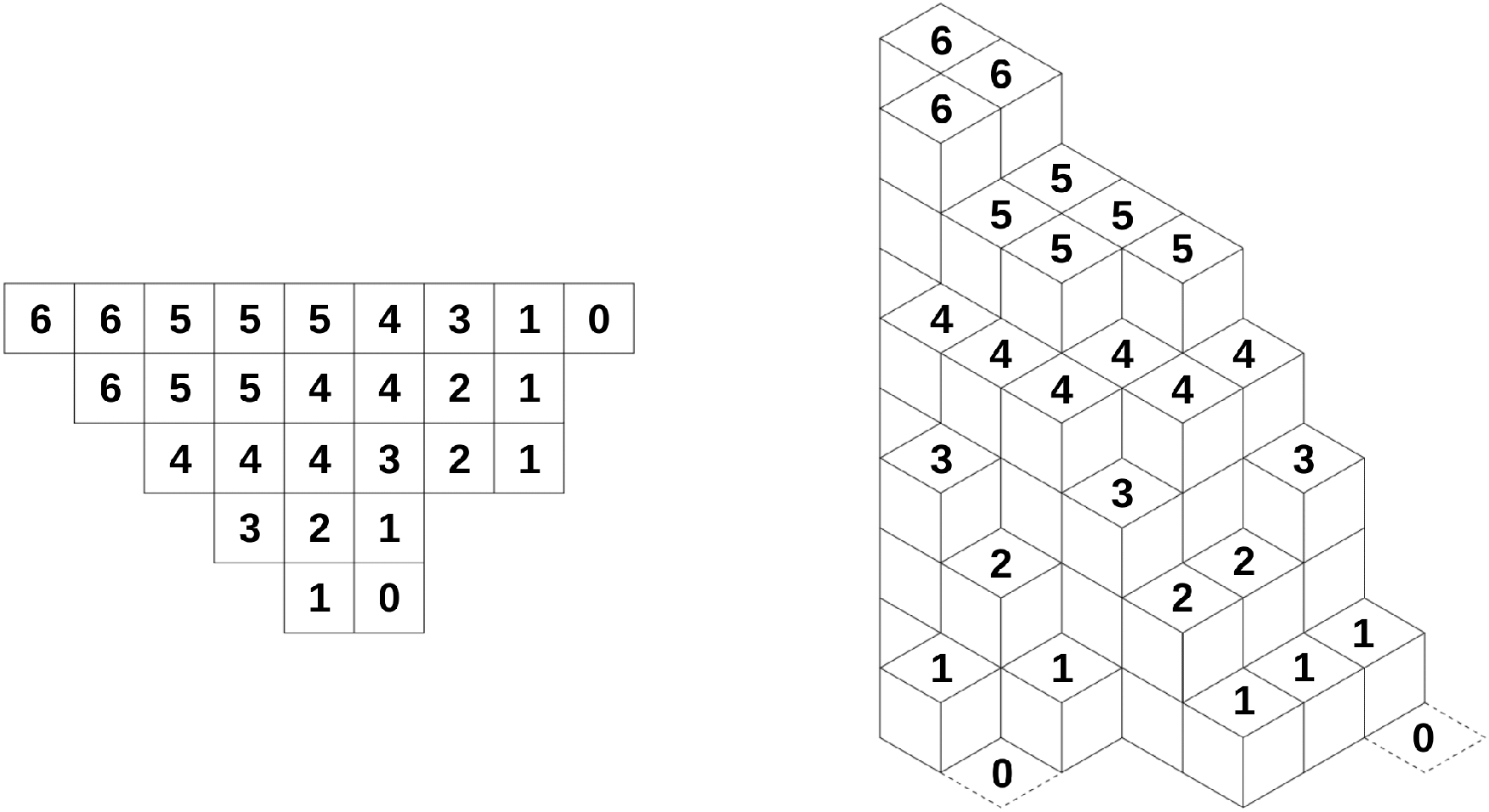}
    \caption{A shifted plane partition in $SPP(m;\lambda_{st})$, where $m=6$ and $\lambda_{st}=(9,7,6,3,2)$ (left) and its 3-D realization as a stack of unit cubes (right). The volume of this shifted plane partition is $87$.}
    \label{fdf}
\end{figure}

Recall that a partition $\lambda=(\lambda_1,\ldots,\lambda_k)$ is \textit{strict} if its entries satisfy the strict inequality $\lambda_1>\ldots>\lambda_k$. When the partition is strict, we use the notation $\lambda_{st}$. Given a strict partition $\lambda_{st}=(\lambda_1,\ldots,\lambda_k)$, its \textit{shifted Young diagram} is a collection of boxes that has $\lambda_i$ boxes in $i$-th row where $i$-th row is indented by $(i-1)$ unit. Then, a \textit{shifted plane partition of shifted shape $\lambda_{st}$} is an array of nonnegative integers $\pi=(\pi_{i,j})_{1\leq i\leq k, i\leq j\leq\lambda_{i}+i-1}$ such that $\pi_{i,j}\geq \pi_{i+1,j}$ and $\pi_{i,j}\geq\pi_{i,j+1}$ for all $i$, $j$ and its volume is again the sum of all its entries and denoted by the same notation $|\pi|\coloneqq\sum_{i=1}^{k}\sum_{j=i}^{\lambda_{i}+i-1}\pi_{i,j}$ (see the left picture in Figure \ref{fdf}). Note that a shifted plane partition can also be realized as a stack of unit cubes, as shown in the right picture in Figure \ref{fdf}. For any strict partition $\lambda_{st}$ and nonnegative integer $m$, the set of shifted plane partitions of shifted shape $\lambda_{st}$ with largest entry no greater than $m$ is denoted by $SPP(m;\lambda_{st})$. The cardinality and the volume generating function of $SPP(m;\lambda_{st})$ are denoted by $|SPP(m;\lambda_{st})|$ and $|SPP(m;\lambda_{st};q)|\coloneqq\sum_{\pi\in SPP(m;\lambda_{st})}q^{|\pi|}$, respectively. Lastly, for any strict partition $\lambda_{st}=(\lambda_{1},\ldots,\lambda_{k})$, consider its Young diagram and symmetrize it across the diagonal through the centers of the left-most boxes in each row. What we get is a Young diagram of a certain symmetric partition, and we denote this partition by $\overline{\lambda_{st}}$. More precisely, $\overline{\lambda_{st}}$ is a symmetric partition whose Young diagram has $\lambda_{i}+i-1$ boxes in $i$-th row for $1\leq i\leq k$ and the size of its Durfee square\footnote{A \textit{Durfee square} of a Young diagram is the largest square that fits inside the Young diagram.} is $k$ (these two properties uniquely determine $\overline{\lambda_{st}}$).

\begin{figure}
    \centering
    \includegraphics[width=1.0\textwidth]{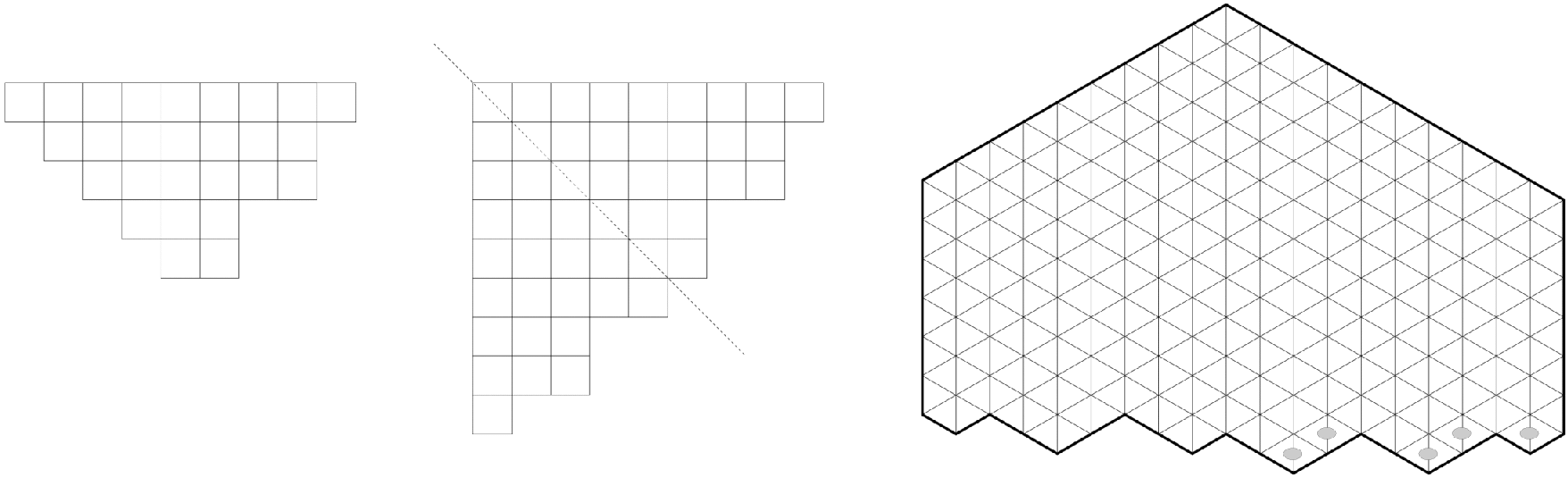}
    \caption{The Young diagrams of $\lambda_{st}$ (left) and $\overline{\lambda_{st}}$ (center) and the region $\widetilde{A}(m;\lambda_{st})$ (right), where $\lambda_{st}=(9,7,6,3,2)$ and $m=6$. Theorem \ref{tdc} states that to find the cardinality of $\mathit{SPP}(m;\lambda_{st})$ and $\mathit{PP}_{sym}(m;\overline{\lambda_{st}})$, it suffices to find the tiling generating function of $\widetilde{A}(m;\lambda_{st})$.}
    \label{fdg}
\end{figure}

We provide two results here. The first result states that to find $|PP_{sym}(m;\overline{\lambda_{st}})|$ or $|SPP(m;\lambda_{st})|$, one can instead consider (weighted) enumeration of lozenge tilings of $\widetilde{A}(m;\lambda_{st})$ (see the pictures in Figure \ref{fdg}).

\begin{thm}\label{tdc}
    For any positive integer $m$ and a strict partition $\lambda_{st}$ with $k$ parts, $|\mathit{SPP}(m;\lambda_{st})|$ and $|\mathit{PP}_{sym}(m;\overline{\lambda_{st}})|$ satisfy
    \begin{equation}\label{tdp}
        |\mathit{SPP}(m;\lambda_{st})|^2=|\mathit{PP}_{sym}(m;\overline{\lambda_{st}})|^2=2^{k}\M(\widetilde{A}(m;\lambda_{st})).
    \end{equation}
\end{thm}
\begin{proof}
   
    The proof is based on bijections among $\mathit{SPP}(m;\lambda_{st})$, $\mathit{PP}_{sym}(m;\overline{\lambda_{st}})$, and the set of lozenge tilings of $A(m;\lambda_{st})$ (this bijection is not new. This can be found, for example, in \cite{HL}). The correspondence is as follows. Given a shifted plane partition in $\mathit{SPP}(m;\lambda_{st})$, reflect the shifted plane partition along the main diagonal and get a symmetric plane partition in $\mathit{PP}_{sym}(m;\overline{\lambda_{st}})$. Note that this correspondence is reversible. More precisely, given a symmetric plane partition in $\mathit{PP}_{sym}(m;\overline{\lambda_{st}})$, get rid of all the entries strictly below the main diagonal and what we get is a shifted plane partition in $\mathit{SPP}(m;\lambda_{st})$. One can readily see that these maps are inverse to each other. Therefore, $\mathit{SPP}(m;\lambda_{st})$ and $\mathit{PP}_{sym}(m;\overline{\lambda_{st}})$ are in bijection (see the two pictures on the top in Figure \ref{fdh}).

\begin{figure}
    \centering
    \includegraphics[width=0.9\textwidth]{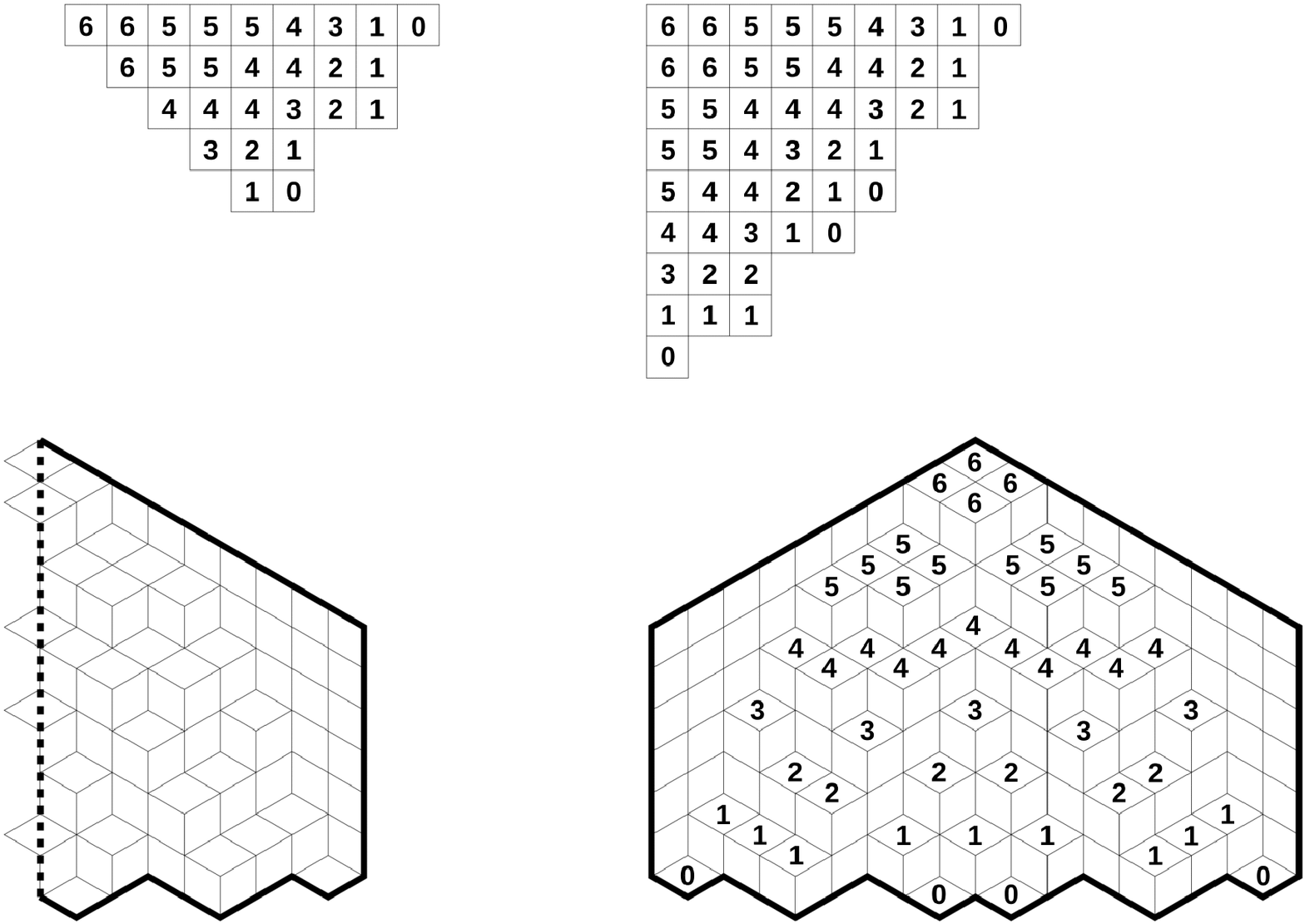}
    \caption{The correspondence among a shifted plane partition of shifted shape in $SPP(m;\lambda_{st})$ (top-left), a symmetric plane partition of symmetric shape in $PP_{sym}(m;\overline{\lambda_{st}})$ (top-right) and its 3-D realization (bottom-right), and a lozenge tilings of the mirror image of $A(m;\lambda_{st})$ (bottom left), where $\lambda_{st}=(9,7,6,3,2)$ and $m=6$.}
    \label{fdh}
\end{figure}

    Now, consider a symmetric plane partition in $\mathit{PP}_{sym}(m;\overline{\lambda_{st}})$. As mentioned earlier, one can realize it as a stack of unit cubes, and using David and Tomei's bijection \cite{DT}, one can then realize it as a vertically symmetric lozenge tiling of a certain region (see the two pictures on the right in Figure \ref{fdh}). Since vertically symmetric lozenge tilings are uniquely determined by lozenges strictly right to the symmetry axis and lozenges crossing the symmetry axis, we discard the lozenges on the left side of the symmetry axis (we keep the lozenges crossing the symmetry axis. See the two picture on the bottom in Figure \ref{fdh}). The remaining lozenges form a lozenge tiling of the mirror image of $A(m;\lambda_{st})$, where its left side is a free boundary. One can readily check that this correspondence is reversible and thus $\mathit{PP}(m;\lambda_{st})$, $\mathit{SPP}(m;\overline{\lambda_{st}})$, and the set of lozenge tilings of $A(m;\lambda_{st})$ are in bijection. Therefore, combining it with Theorem \ref{tca}, we have
    \begin{equation}\label{tdq}
        |\mathit{SPP}(m;\lambda_{st})|^2=|\mathit{PP}_{sym}(m;\lambda_{st})|^2=\M_f(A(m;\lambda_{st}))^2=2^{k}\M(\widetilde{A}(m;\lambda_{st})).
    \end{equation}
This completes the proof.
\end{proof}

The second result gives determinant formulas for the volume generating functions of symmetric plane partitions of a symmetric shape ($|PP_{sym}(m;\overline{\lambda_{st}};q)|$) and the shifted plane partitions of a shifted shape ($|SPP(m;\lambda_{st};q)|$). It uses weight-preserving bijections among shifted plane partitions, lozenge tilings, and families of nonintersecting paths. These bijections were introduced earlier in the proof of Theorem \ref{tca} and Theorem \ref{tdc} without weight on them.

For any nonnegative integer $n$, the \textit{$q$-analogue of $n$} is $[n]_q\coloneqq\frac{1-q^n}{1-q}$ and the \textit{$q$-analogue of $n!$} is $[n]_{q}!\coloneqq\prod_{i=1}^{n}[i]_q!$. Also, for integers $n$ and $k$, the \textit{$q$-analogue of the binomial coefficient $\binom{n}{k}$} is $\begin{bmatrix}
    n\\
    k
\end{bmatrix}_q\coloneqq\frac{[n]_q!}{[k]_q![n-k]_q!}$ if $0\leq k\leq n$ and $0$ otherwise. Note that we cannot replace $q$ by $1$ in $[n]_q$, $[n]_{q}!$, or
$\begin{bmatrix}
    n\\
    k
\end{bmatrix}_q$ because the term $(1-q)$ in the denominator vanishes. However, one can check that if we take the limit $q\rightarrow 1$, then their limits become $n$, $n!$, and $\binom{n}{k}$, respectively. For any shifted plane partition of shifted shape $\lambda_{st}$, we define its \textit{$(q,t)$-weight} as $q^{\sum_{i<j}\pi_{i,j}}t^{\sum_{i}\pi_{i,i}}$. In terms of its 3-D realization as a stack of unit cubes, the exponent of $t$, $\sum_{i}\pi_{i,i}$, is the number of unit cubes on the main diagonal and the exponent of $q$, $\sum_{i<j}\pi_{i,j}$, is the number of unit cubes not on the main diagonal. Using this new weight, we define a \textit{$(q,t)$-generating function of $\mathit{SPP}(m;\lambda_{st})$} as follows: $|\mathit{SPP}(m;\lambda_{st};q,t)|\coloneqq\sum_{\pi\in \mathit{SPP}(m;\lambda_{st})}q^{\sum_{i<j}\pi_{i,j}}t^{\sum_{i}\pi_{i,i}}$. The main result of the following theorem gives a determinant formula for this $(q,t)$-generating function $|\mathit{SPP}(m;\lambda_{st};q,t)|$. Once we prove it, determinant formulas for the volume generating functions of $\mathit{SPP}(m;\lambda_{st})$ and $\mathit{PP}_{sym}(m;\overline{\lambda_{st}})$ can be easily deduced.

\begin{thm}\label{tdd}
    For any positive integer $m$ and a strict partition $\lambda_{st}=(\lambda_1,\ldots,\lambda_k)$, the $(q,t)$-generating function of shifted plane partition of shifted shape $\lambda_{st}$ with largest entry no greater than $m$ satisfies the following determinant formula.
    \begin{equation}\label{tdr}
        |\mathit{SPP}(m;\lambda_{st};q,t)|^2=\det \Big[\M(q,t)U_{m+k}\M(q,t)^T\Big],
    \end{equation}
    Where $\M(q,t)=\Bigg[t^{m+i-j}\begin{bmatrix}
        \lambda_i-1+m+i-j\\
        m+i-j
    \end{bmatrix}_q\Bigg]_{1\leq i\leq k, 1\leq j\leq m+k}$ is a $k\times (m+k)$ matrix and $U_{m+k}=[d_{i,j}]_{1\leq i,j\leq m+k}$ is an upper triangular matrix defined by $d_{i,j}=
    \begin{cases}
    2, & \text{if $i<j$}\\
    1, & \text{if $i=j$}\\    
    0, & \text{if $i>j$}
    \end{cases}$.
    
    In particular, the volume generating functions of $\mathit{SPP}(m;\lambda_{st})$ and $\mathit{PP}_{sym}(m;\lambda_{st})$ satisfy
    \begin{equation}\label{tds}
        |\mathit{SPP}(m;\lambda_{st};q)|^2=\det \Big[\M(q,q)U_{m+k}\M(q,q)^T\Big]
    \end{equation}
    and
    \begin{equation}\label{tdt}
        |\mathit{PP}_{sym}(m;\lambda_{st};q)|^2=\det \Big[\M(q^2,q)U_{m+k}\M(q^2,q)^T\Big],
    \end{equation}
    respectively, where the two matrices $\M(q,q)=\Bigg[q^{m+i-j}\begin{bmatrix}
        \lambda_i-1+m+i-j\\
        m+i-j
    \end{bmatrix}_q\Bigg]_{1\leq i\leq k, 1\leq j\leq m+k}$ and $\M(q^2,q)=\Bigg[q^{m+i-j}\begin{bmatrix}
        \lambda_i-1+m+i-j\\
        m+i-j
    \end{bmatrix}_{q^2}\Bigg]_{1\leq i\leq k, 1\leq j\leq m+k}$ are the specializations of $\M(q,t)$ and $U_{m+k}$ is the same matrix as in \eqref{tdr}.
\end{thm}
\begin{proof}
We first recall the two bijections presented in the proof of Theorems \ref{tca} and \ref{tdc}. We have a bijection between $\mathit{SPP}(m;\lambda_{st})$ and the set of lozenge tilings of the mirror image of $A(m;\lambda_{st})$, and this was explained in the proof of Theorem \ref{tdc} (see the left two pictures in Figure \ref{fdi}. At this point, ignore the labels in the pictures). We then have a bijection between the set of lozenge tilings of $A(m;\lambda_{st})$ and the set of families of $k$ nonintersecting paths on $\mathbb{Z}^2$, which was presented in the proof of Theorem \ref{tca} (see the right two pictures in Figure \ref{fdi}. The pictures were reflected, as we want to compose this correspondence with the one involving $\mathit{SPP}(m;\lambda_{st})$. Again, ignore the labels). Combining these two bijections, we obtain a bijection between $\mathit{SPP}(m;\lambda_{st})$ and the set of families of $k$ nonintersecting paths on $\mathbb{Z}^2$ with prescribed starting and ending points: the starting points are $u_i=(\lambda_i-1, k-i)$ for $i\in[k]$ and the ending points are $v_j=(0,m+k-j)$ for $j\in[m+k]$ and all edges are oriented toward the top and left. We now assign weight to unit segments in $\mathbb{Z}^2$ lattice as follows, so that the $(q,t)$-weight of the shifted plane partition is encoded in the nonintersecting paths:
\begin{itemize}
    \item Every horizontal unit segment is weighted by $1$. 
    \item Every vertical unit segment on $x=a$ is weighted by $q^{a}t$.
\end{itemize}
To see why this weight makes the bijection weight-preserving, we consider a shifted plane partition. Given a 3-D realization of the shifted plane partition as a stack of unit cubes, note that each level of the stack forms a shifted Young diagram. We read the size of each row and then mark the monomial $q^{\textbf{(the size of the row)}-1}t$ at the right side of the rightmost unit cube on each row (see the left picture in Figure \ref{fdi}). Since each row contain exactly one unit cube that is on the main diagonal (namely, the leftmost one), it is clear that the product of these monomials gives $q^{\sum_{i<j}\pi_{i,j}}t^{\sum_{i}\pi_{i,i}}$, the $(q,t)$-weight of the shifted plane partition. One can then readily see that the weight of the corresponding family of $k$-tuples of lattice paths on $\mathbb{Z}^2$ with the aforementioned weights on its edges is the same, and thus the bijection is weight-preserving (see Figure \ref{fdi}). Therefore, our task is reduced to finding a generating function for the families of $k$-tuples of nonintersecting paths on $\mathbb{Z}^2$ starting at $(u_i)_{i\in[k]}$ and ending at $(v_{j})_{j\in[m+k]}$.

\begin{figure}
    \centering
    \includegraphics[width=1.0\textwidth]{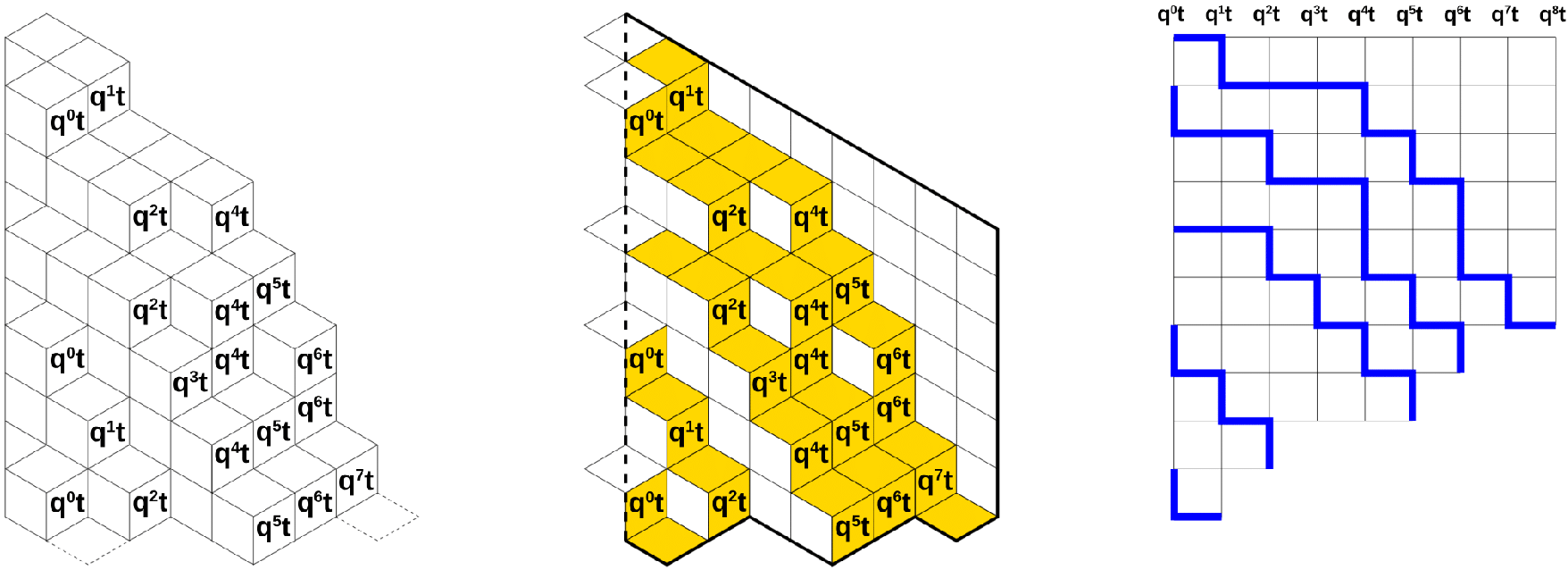}
    \caption{A 3-D realization of a shifted plane partition in $SPP(m;\lambda_{st})$ is on the left, where $m=6$ and $\lambda_{st}=(9,7,6,3,2)$, and the monomials $q^{\textbf{(the size of the row)}-1}t$ are marked in each row. In the middle picture, the corresponding lozenge tiling of the mirror image of $A(m,\lambda_{st})$ is presented. In that picture, the weights of the lozenges are marked (the ones with no mark are weighted by $1$). The picture on the right presents the corresponding family of nonintersecting paths on $\mathbb{Z}^2$. The vertical unit segments below the label $q^kt$ are weighted by $q^kt$, and all horizontal unit segments are weighted by $1$.}
    \label{fdi}
\end{figure}

To apply Theorem \ref{tca}, we need to find a generating function for lattice paths starting at $u_i$ and ending at $v_j$: this will be the $(i,j)$-entry of the matrix $\M(q,t)$. To find it, we need to show the following claim.

\textbf{Claim.} Consider a lattice $\mathbb{Z}^2$ where all the edges are oriented toward the top and left and weighted as described above. For two points $(a,b), (c,d)\in\mathbb{Z}^2$,
\begin{equation}\label{tdu}
    \GF[\mathscr{P}((a,b),(c,d))]=
    \begin{cases}
        q^{c(d-b)}t^{(d-b)}\begin{bmatrix}
            (a-c)+(d-b)\\
            (d-b)
        \end{bmatrix}_q&\text{if $c\leq a$ and $b\leq d$}\\
        0, &\text{otherwise}
    \end{cases}.
\end{equation}
The above claim can be proved as follows.

\begin{itemize}
    \item When $c>a$ or $b>d$, then there is no lattice path connecting $(a,b)$ and $(c,d)$, so the generating function is $0$.
    \item When $c\leq a$ and $b=d$, then there is only one lattice path that consists of horizontal edges. This lattice path has weight $1$ and it agrees with \eqref{tdu}.
    \item When $a=c$ and $b\leq d$, then there is again only one lattice that consists of $(d-b)$ vertical edges. Since each edge has weight $q^{a}t$, the lattice path has weight $q^{a(d-b)}t^{(d-b)}$ and it agrees with \eqref{tdu}.
    \item When $c<a$ and $b<d$, we can prove the claim using an induction on $(a-c)+(d-b)$. To proceed with the induction, one can use the recurrence
    \begin{equation}\label{tdv}
        \GF[\mathscr{P}((a,b),(c,d))]=q^{a}t\GF[\mathscr{P}((a,b+1),(c,d))]+\GF[\mathscr{P}((a-1,b),(c,d))],
    \end{equation}
    which can be obtained from the fact that the lattice path from $(a,b)$ to $(c,d)$ can be partitioned according to the first step: if the first step is vertical (or horizontal), the edge has weight $q^{a}t$ (or $1$) and the remaining steps form a lattice path joining $(a,b+1)$ and $(c,d)$ (or $(a-1,b)$ and $(c,d)$). Since the formula in \eqref{tdu} satisfies the recurrence \eqref{tdv} and $(a-c)+(d-b)>(a-c)+(d-(b+1))=((a-1)-c)+(d-b)$, \eqref{tdu} can be proved by induction. We leave the details of the verification of the induction step to the readers.
\end{itemize} 
Therefore, the generating function of lattice paths from $u_i=(\lambda_i-1,k-i)$ to $v_j=(0,m+k-j)$ is $t^{m+i-j}\begin{bmatrix}
        \lambda_i-1+m+i-j\\
        m+i-j
    \end{bmatrix}_q$
and thus applying Theorem \ref{tca}, we can conclude that the $(q,t)$-generating function of $SPP(m;\lambda_{st})$ is given by $\det\Big[\M(q,t)U_{m+k}\M(q,t)^T\Big]$ where
$\M(q,t)=\Bigg[t^{m+i-j}\begin{bmatrix}
        \lambda_i-1+m+i-j\\
        m+i-j
    \end{bmatrix}_q\Bigg]$
is a $k\times (m+k)$ matrix and $U_{m+k}$ is an upper triangular matrix defined in the statement of Theorem \ref{tdd}. This completes the proof of \eqref{tdr}. 

To prove \eqref{tds} and \eqref{tdt}, recall that
\begin{equation}\label{tdw}
    |\mathit{SPP}(m;\lambda_{st};q,t)|\coloneqq\sum_{\pi\in \mathit{SPP}(m;\lambda_{st})}q^{\sum_{i<j}\pi_{i,j}}t^{\sum_{i}\pi_{i,i}}.
\end{equation}
Note that the volume generating function of $\mathit{SPP}(m;\lambda_{st})$ can be written as follows.
\begin{equation}\label{tdx}
    |\mathit{SPP}(m;\lambda_{st};q)|=\sum_{\pi\in \mathit{SPP}(m;\lambda_{st})}q^{\sum_{i\leq j}\pi_{i,j}}=\sum_{\pi\in \mathit{SPP}(m;\lambda_{st})}q^{\sum_{i<j}\pi_{i,j}}q^{\sum_{i}\pi_{i,i}}.
\end{equation}
Similarly, using the fact that $\pi_{i,j}=\pi_{j,i}$ for all $\pi\in PP_{sym}(m,\overline{\lambda_{st}})$, the volume generating function of $PP_{sym}(m,\overline{\lambda_{st}})$ can be expressed as follows.

\begin{equation}\label{tdy}
\begin{aligned}
    |PP_{sym}(m,\overline{\lambda_{st}};q)|=\sum_{\pi\in PP_{sym}(m,\overline{\lambda_{st}})}q^{\sum_{i,j}\pi_{i,j}}=&\sum_{\pi\in PP_{sym}(m,\overline{\lambda_{st}})}q^{\sum_{i>j}\pi_{i,j}}q^{\sum_{i=j}\pi_{i,j}}q^{\sum_{i<j}\pi_{i,j}}\\
    =&\sum_{\pi\in PP_{sym}(m,\overline{\lambda_{st}})}q^{2\sum_{i<j}\pi_{i,j}}q^{\sum_{i=j}\pi_{i,j}}\\
    =&\sum_{\pi\in SPP(m,\lambda_{st})}q^{2\sum_{i<j}\pi_{i,j}}q^{\sum_{i=j}\pi_{i,j}},
\end{aligned}
\end{equation}
where in the last equation, we use the bijection between $PP_{sym}(m,\overline{\lambda_{st}})$ and $SPP(m,\lambda_{st})$ introduced earlier. One can compare \eqref{tdw}, \eqref{tdx}, and \eqref{tdy} and notice that the volume generating functions of $\mathit{SPP}(m;\lambda_{st})$ and $PP_{sym}(m,\overline{\lambda_{st}})$ can be obtained from the $(q,t)$-generating function of $\mathit{SPP}(m;\lambda_{st})$ by setting $(q,t)=(q,q)$ and $(q,t)=(q^2,q)$, respectively. Therefore, \eqref{tds} and \eqref{tdt} follow for \eqref{tdr} by setting $(q,t)=(q,q)$ and $(q,t)=(q^2,q)$ in \eqref{tdr}, respectively. This completes the proof.
\end{proof}

\bibliography{bibliography}{}
\bibliographystyle{abbrv}
  
\end{document}